\documentclass[11pt,a4paper]{amsart}
\usepackage{amsmath,amssymb,amscd}
\usepackage{hyperref}
\usepackage{breakurl}
\newtheorem{theorem}{Theorem}
\newtheorem{proposition}{Proposition}
\newtheorem{corollary}{Corollary}
\newtheorem{lemma}{Lemma}
\newcommand{\T}{{\rm pc}}
\newcommand{\M}{{\rm meas}}
\newcommand{\bs}{\boldsymbol}

\title[Entropy and diffraction of the $k$-free points]
{Entropy and diffraction of the $k$-free points in $n$-dimensional lattices}

\author[Peter A.~B.~Pleasants]{Peter A.~B.~Pleasants~\dag}
\address{Department of Mathematics, University of Queensland,
Brisbane, QLD 4072, Australia}

\author{Christian Huck}
\address{Fakult\"at f\"ur Mathematik, Universit\"at Bielefeld,
  Postfach 100131, Bielefeld, Germany}
\email{huck@math.uni-bielefeld.de}

\begin{document}

\setlength{\unitlength}{1mm}

\begin{abstract}
We consider the $k$th-power-free points in $n$-dimensional lattices
and explicitly calculate their entropies and diffraction
spectra. This is of particular interest since these sets have holes
of unbounded inradius.
\end{abstract}

\maketitle

\section{Introduction}
In \cite{BMP} the diffraction properties of the visible points and the
$k$th-power-free numbers were studied and it was shown that these sets
have positive, pure-point, translation-bounded diffraction spectra with countable,
dense support. The interest of this lay in the fact that these sets
fail to be Delone sets: they are uniformly discrete (subsets of lattices, in fact)
but not relatively dense. The lack of relative denseness means that these
sets have arbitrarily large ``holes" and hence are not repetitive in the
sense of \cite{LP}. It is of interest to ask for more precise information
about the irregularity of these sets, and Lenz \cite{L} has asked what
their entropy is.

There are two kinds of entropy commonly associated with arrays of symbols
(of which subsets of lattices are a particular case): \emph{patch-counting
entropy} which is defined simply by counting patches and depends only on
the adjacency relation between sites, not on any metric of the ambient
space; and \emph{measure entropy} which is defined in terms of the
frequency of occurrence of patches in space. The patch-counting entropy
is an upper bound for the measure entropy, whatever measure is used.
We show that the sets considered here have measure entropy zero
(relative to a canonically constructed measure) but
positive patch-counting entropy, contrasting with regular model sets~\cite{Moo1},
for which both entropies are zero~\cite{BLR}. In \cite{BMP}, a model set construction for the visible points and the
$k$th-power-free numbers was described, with the internal spaces
adelic, instead of Euclidean as in more usual cut-and-project sets. In this
construction, the boundaries of the windows have positive measure, however,
so they are not regular model sets.

In Section~\ref{defns}, we define patch-counting and measure
entropies, while in Section~\ref{kfree} we define the set of
$k$-free points, whose entropies we investigate, and
show that they possess patch frequencies which can be explicitly
calculated in terms of infinite products.  This is just a mild
generalization to the case of lattices other than $\mathbb Z$
of the results of Mirsky \cite{Mir2} on $k$th-power-free
integers. To keep the route to our main results as clear
 as possible we have been content with weak error terms
 in Section~\ref{kfree}, but for the record we show in
Section~\ref{errorterm} how error terms like those in \cite{Mir1}
carry over to the general case. Section~\ref{examples} gives some
examples of patch frequencies and Section~\ref{entropy} completes
the calculation of the entropies, with the aid of a key lemma
(for the measure entropy case) that gives a small upper bound
for the frequencies of the great majority of
patches. In Section~\ref{variational}, we give a short discussion of
the 
variational principle, which relates the two kinds of entropy. In
Section~\ref{diffraction}, we demonstrate how the results in~\cite{BMP}
on the diffraction spectra of the
$k$th-power-free integers and visible lattice points carry over to the general
case. 

For the special case of square-free numbers (resp., $k$th-power-free numbers), some of our 
results were found independently by employing alternative
methods from the theory of dynamical systems by Cellarosi and
Sinai~\cite{CS}, Cellarosi and Vinogradov~\cite{CV} and by
Sarnak~\cite{Sarnak}. Furthermore, these references also contain results on the ergodic properties of the underlying
invariant measures that go beyond what we cover here.

In the course of the paper, we need to call on a number of standard
results in number theory, which for convenience we have collected
in an appendix (whose equation numbers carry a prefix `A').
\medskip

Peter A.~B.~Pleasants gave me (CH) an early draft of this paper already in
2006. After his untimely death in 2008, 
Michael Baake asked me to finish the manuscript. At that time, it already
contained the entire calculation of the entropies (Sections 1--5). Moreover, Peter had
planned two further sections, one on improved error terms and one on a
model set 
construction including the sets in
question together with an upper bound for the topological entropies
that is intrinsic to the corresponding window. While the former is now
included (Section 8), the latter is still 
work in progress. Instead, the text
now has two additional sections, one on a variational principle
(Section~\ref{variational}) and one on the diffraction of the
sets studied here (Section~\ref{diffraction}).

\section{Definitions of entropy}\label{defns}
Let $X$ be a subset of a lattice $\Lambda$ in $\mathbb R^n$.
Given a radius $\rho>0$ and a point ${\bs t}\in\Lambda$,
the \emph{$\rho$-patch} of $X$ at $\bs t$ is
\[(X-{\bs t})\cap B_\rho({\bs 0}),\]
the translation to the origin of the part of $X$ within a distance
$\rho$ of $\bs t$. We denote by $\mathcal A(\rho)$ the set of all
$\rho$-patches of $X$ and by $N(\rho)=|\mathcal A(\rho)|$ the number of
distinct $\rho$-patches of $X$. Then the \emph{patch-counting entropy} of $X$ is
\begin{equation}\label{top}
h_\T(X):=\lim_{\rho\to\infty}\frac{\log_2N(\rho)}{\rho^nv_n},
\end{equation} where $v_n$ is the volume of an $n$-dimensional ball of
radius 1, i.e.\ $v_n=\pi ^{n/2}/\Gamma(1+\frac{n}{2})$ (so that the
denominator is the volume of the open ball $B_\rho({\bs 0})$).
It can be shown by a subadditivity argument that this limit exists
for every $X\subset\Lambda$.  In~\cite[Theorem 1 and Remark 2]{BLR} Baake, Lenz and Richard show that,
for the dynamical system of coloured Delone sets of finite local complexity, the patch-counting
entropy coincides with the topological entropy; see Section~\ref{variational} for more on the natural dynamical system associated
with a subset $X$ of $\Lambda$ and the $k$-free points in particular.

To describe measure entropy, we must take into account densities of subsets of a
lattice.  If $Y\subset\Lambda$, its \emph{density} $\delta(Y)$ is defined by
\begin{equation}\label{densY}
\delta(Y):=\lim_{R\to\infty}\frac{|Y\cap B_R(\bs0)|}{R^nv_n},
\end{equation}
when the limit exists; cf.~\cite{BMP} for related ways of defining
densities of discrete point sets.  In cases where the limit does not exist, we can still
define an \emph{upper density}, $\bar\delta(Y)$ and a \emph{lower density},
$\underline\delta(Y)$, by replacing the limit in \eqref{densY} by
$\limsup$ or $\liminf$.  The \emph{frequency}, $\nu(\mathcal P)$,
of a $\rho$-patch $\mathcal P$ of $X$ is defined by
\begin{equation}\label{frequency}
\nu(\mathcal P):=\delta(\{\bs t\in\Lambda:\mbox{the $\rho$-patch of $X$ at $\bs t$ is $\mathcal P$}\}),
\end{equation}
when this density exists. In the absence of a well defined density, 
we can still define an \emph{upper frequency}, $\bar\nu(\mathcal P)$
and a \emph{lower frequency}, $\underline\nu(\mathcal P)$, by replacing
$\delta$ by $\bar\delta$ or $\underline\delta$. The \emph{measure entropy}
of $X$, which can be thought of as corresponding to the metric entropy of
a dynamical system, is now defined by
\begin{equation}\label{met}
h_\M(X):=\lim_{\rho\to\infty}
\frac{1}{\rho^nv_n}\sum_{\mathcal P\in\mathcal A(\rho)}\!\!\!-\nu(\mathcal P)\log_2\nu(\mathcal P),
\end{equation}
with the convention that $\nu\log_2\nu=0$ when $\nu=0$; see
Section~\ref{variational} for details.  It is
defined when every patch of $X$ has a well defined frequency,
in which case a subadditivity argument again shows that the limit
exists. Since $\nu\log_2\nu$ is a convex function of $\nu$,
the sum does not decrease if we replace the $\nu(\mathcal P)$'s
by their average value, $1/N(\rho)$, to make the right side
the same as the right side of \eqref{top}. Hence
\[h_\M(X)\le h_\T(X).\]

As a simple example where these entropies differ, consider
the binary sequence consisting of the binary numbers in order
(0, 1, 10, 11, 100, \dots) with $n,n+1$ separated by $n$ 1's:
\[\underline{0}\hspace{.5mm}\underline{1}1\underline{10}11
\underline{11}111\underline{100}1111\underline{101}11111
\underline{110}111111\underline{111}1111111\underline{1000}1111\ldots.\]
Evidently, there are very few 0's to contribute variety here.
In fact the sequence of 0's has density zero, and consequently
any finite word that is not all 1's has frequency zero.
So $h_\M=0$.  But since there are $2^l$ possible words of length
$l$ and every word occurs somewhere, $h_\T=1$.

In general, $h_\T$ is a combinatorial function of the set of finite
configurations that occur, while $h_\M$ is a geometric function of
an infinite configuration and can differ among different infinite
configurations built up from the same set of finite ones, with $h_\T$
being an upper bound for the possible values it can take. Of the two
entropies, $h_\M$ would appear to carry more physical significance.

More generally, if we have a pattern formed by labelling the points of
$\Lambda$ with letters from an $a$-letter alphabet then we can again
define \mbox{$\rho$-patches}, and the entropies of the pattern are given
by \eqref{top} and \eqref{met} with 2 replaced by $a$ as the base of
logarithms. A subset of $\Lambda$ corresponds to a 2-letter labelling
indicating whether or not a site is occupied. The reason for the patch
volume in the denominator and for the choice of base of logarithms is
to normalize so that the integer lattice with random labelling has both
entropies 1.

There are various ways in which the definition of measure entropy might
be extended to sets $X$ for which not all patch frequencies exist.
A first step would be to replace the sum in \eqref{met} by
\[
\lim_{R\to\infty}\sum_{\mathcal P\in\mathcal A(\rho)}
-\frac{|L(\mathcal P)\cap B_R(\bs0)|}{R^nv_n}
\log_2\left(\frac{|L(\mathcal P)\cap B_R(\bs0)|}{R^nv_n}\right),
\]
where $L(\mathcal P)$ is the set appearing in \eqref{frequency}.  This
delays taking the limit, so that it has a chance of existing even when
some individual patch frequencies may fail to exist. We shall not need such
extensions here, however, since Theorem~\ref{density} below guarantees
that, for the sets studied in this paper, all patch frequencies exist.

\section{$k$-free points}\label{kfree}

As a convenient context for our results, we shall use the set $V=V(\Lambda,k)$
of $k$-free points of a lattice $\Lambda$ in $\mathbb R^n$.  For a point
$\bs l\ne\bs 0$ in $\Lambda$ define its \emph{$k$-content}, $c_k(\bs l)$,
to be the largest integer $c$ such that $\bs l\in c^k\Lambda$.  Then
$c_k(\bs l)$ is also the least common multiple of the numbers $d$ with
$d^{-k}{\bs l}\in\Lambda$, i.e.\ $d^{-k}{\bs l}\in\Lambda$ if and only
if $d\mid c_k(\bs l)$. For consistency and convenience, we define
$c_k(\bs 0)=\infty$, with the understanding that $d\mid\infty$ for any
number 
$d$. The \emph{$k$-free points}, $V=V(\Lambda,k)$, of
$\Lambda$ are the points with $c_k(\bs l)=1$. One can see that $V$ is
non-periodic, i.e.\ $V$ has no nonzero translational symmetries. As particular cases we have
the visible points of $\Lambda$ (with $n\ge2$ and $k=1$), treated in \cite{BMP},
and the $k$-free integers (with $\Lambda=\mathbb Z$), treated in \cite{BMP},
\cite{Mir1}  and \cite{Mir2}.  The more general context has the advantage
of avoiding duplication of near-identical proofs.  When $n=k=1$, $V$ consists
of just the two points of $\Lambda$ closest to $\bs 0$ on either side, and we
exclude this trivial case. Since $\Lambda$ is a free Abelian group of rank $n$, its automorphism
group, $\operatorname{Aut}(\Lambda)$, is isomorphic to the matrix
group ${\rm GL}(n,\mathbb Z)$. Explicit
isomorphisms can be found by taking coordinates with respect to any
basis of $\Lambda$. Since the action of ${\rm GL}(n,\mathbb Z)$ on $\Lambda$ preserves 
$k$-content, the $k$-free points $V$ are invariant under the action of
${\rm GL}(n,\mathbb Z)$.

\begin{proposition}\label{holes}
$V$ is uniformly discrete, but has arbitrarily large holes. Moreover,
for any $r>0$, there is a set of holes in $V$ of inradius at least $r$
whose centres have positive density.
\end{proposition}
\begin{proof}
Since $V\subset\Lambda$, the uniform discreteness is trivial. Now let
$C=\{\bs a_1,\dots,\bs a_s\}$ be any finite configuration of points in
$\Lambda$ (e.g., all points in a ball or a cube). Choose $s$
integers $m_1,\dots,m_s>1$ that are pairwise coprime (e.g., the first
$s$ primes). By~\eqref{CRT}, there is a point $\bs a\in\Lambda$ with
\[
\bs a\equiv -\bs a_i\pmod {m_i^k\Lambda}
\]
for $i=1,\dots,s$. Now for any $\bs x\equiv \bs
a\pmod{m_1^k\cdots m_s^k\Lambda}$ the configuration $C+\bs x=\{\bs a_1+\bs
x,\dots,\bs a_s+\bs x\}$ is congruent, in the geometric sense, to $C$
but no point in $C+\bs x$ is in $V$, since $\bs a _i+\bs x \in
m_i^k\Lambda$ for $i=1,\dots,s$. The points $\bs x$ have density
$1/((m_1\cdots m_s)^{nk}\det(\Lambda))>0$ by~\eqref{LambdaCount} and~\eqref{CRT}.
\end{proof}

For a natural number $P$, we define $V_P=V_P(\Lambda,k)$ to be the set
of points ${\bs l}\in\Lambda\setminus\{\bs0\}$ with $(c_k({\bs l}),P)=1$. Clearly $V_P$
is fully periodic with a lattice of periods that contains $P^k\Lambda$. The
$V_P$'s are partially ordered inversely to the divisibility partial order
on $\mathbb N$, that is, $V_{PQ}\subset V_P$ for all $P,Q$.  In fact,
more precisely, $V_{PQ}=V_P\cap V_Q$.  The intersection of all the $V_P$'s
is $V$, so if $P$ is divisible by all primes up to a large bound $V_P$
can be regarded as a set of ``potentially $k$-free" points.

For a finite subset $\mathcal F$ of $\Lambda$ and a positive
integer $m$, we shall use
\[\mathcal F/m\Lambda\]
to denote the set of cosets of $m\Lambda$ in $\Lambda$ that
are represented in $\mathcal F$.  We also write
\[D(\mathcal F):=\max_{\bs l,\bs m\in\mathcal F}\|\bs l-\bs m\|\]
for the diameter of $\mathcal F$, where $\|\cdot\|$ denotes the
Euclidean norm on $\mathbb R^n$.

Since entropies of sets in $\mathbb R^n$ vary under change of scale
inversely as the $n$th power of the scaling constant, it is sufficient to
consider lattices of determinant 1. (For other lattices the formula for
the entropy of $V$ must simply be divided by the determinant of $\Lambda$.)
We fix the following notation for the rest of this paper:
\begin{quote}
\emph{$\Lambda$ is a lattice of determinant $1$ in $\mathbb R^n$, $\lambda$
is the length of its shortest nonzero vector, $k$ is a natural number 
(with $k\ge2$ if $n=1$) and $V$ is the set of $k$-free points in $\Lambda$.}
\end{quote}

Also, for subsets $X,\mathcal P,\mathcal Q$ of $\Lambda$, with $X$
infinite but $\mathcal P,\mathcal Q$ finite, we define the locator set
\[L(X;\mathcal P,\mathcal Q):=
\{\bs t\in\Lambda:\mathcal P+\bs t\subset X,\ \mathcal Q+\bs t\subset\Lambda\setminus X\}\]
consisting of those lattice translations that locate $\mathcal P$
totally inside $X$ and $\mathcal Q$ totally outside $X$.

The genesis of our proof of positive, but non-maximal, patch-counting
entropy for the visible points is the observation that, of the 4 corners
of any unit square of the integer lattice in the plane, at least one
is invisible (because both its coordinates are even) but each of the
15 possibilities for the visibility or not of the corners, when the
possibility of their all being visible is excluded, can occur, depending
on the position of the square within the lattice.  This is the simplest
example of the fact that, in general, every $\rho$-patch contains an
irreducible minimum of points not in $V$ but for the remaining points
in the patch we can arrange that they are visible or not, independently
of each other by choosing the position of the patch in the lattice.
This leads to an exponentially large number of $\rho$-patches, the
number of which can be estimated quite accurately.

Our aim with the following lemma is to concentrate most of the necessary
inclusion-exclusion arguments into a single result from which ensuing
results can be fairly readily derived. For this reason it has several
parameters ($\mathcal P$, $m$, $\bs m$, $P$ and $\bs x$) and three
components to its error term.  Until the parameters are further
specified, there is no assumption that the error terms are of smaller order than the
main term.  To keep the proof short we have not made the error terms as
small as possible---in Section~\ref{errorterm} we make use
of the technique of \cite{Mir1} to vastly improve the last error term.

\begin{lemma}\label{VPpatches}
Let $\mathcal P$ be a finite subset of $\Lambda$, $m\in\mathbb N$,
$\bs m\in\Lambda$, $P$ be a natural number coprime to $m$ and $\bs x\in\mathbb R^n$.  Then
\[|L(V_P;\mathcal P,\emptyset)\cap(\bs m+m\Lambda)\cap B_R(\bs x)|\]
is estimated by a main term
\begin{equation}\label{main}
\frac{R^nv_n}{m^n}\prod_{p\mid P}\biggl(1-\frac{|\mathcal P/p^k\Lambda|}{p^{nk}}\biggr)
\end{equation}
with error
\begin{equation}\label{error}
O\Bigl(R^{1/k}+R^{n-1}(\min\{\log\log P,\log S\})^{|\mathcal P|}
+\min\{\tau_{|\mathcal P|+1}(P),(S/\lambda)^{|\mathcal P|/k}\}\Bigr),
\end{equation}
where $S:=R+\|\bs x\|+\max_{\bs p\in\mathcal P}\|\bs p\|$, $\tau_r$
is the \mbox{$r$-divisor} function in \eqref{taubound}, and the
$O$-constant depends only on $\Lambda$, $k$ and $\mathcal P$.
\end{lemma}

\begin{proof}
We may clearly assume that $P$ is squarefree.  For each prime $p$
the points $\bs t$ with $\bs t+\mathcal P\subset V_p$ consist of
$p^{nk}-|\mathcal P/p^k\Lambda|$ cosets of $p^k\Lambda$ in $\Lambda$
(those cosets $\bs t+p^k\Lambda$ with $(-\bs t+p^k\Lambda)\cap \mathcal P=\emptyset$).  Clearly $|\mathcal P/p^k\Lambda|=|\mathcal P|$
when $p^k\lambda>D(\mathcal P)$.  Let $Q$ be the product of those prime factors
$p$ of $P$ with $|\mathcal P/p^k\Lambda|<|\mathcal P|$.  By the Chinese
Remainder Theorem \eqref{CRT}, $L(V_Q;\mathcal P,\emptyset)\cap(\bs m+m\Lambda)$ consists of
\[\prod_{p\mid Q}(p^{nk}-|\mathcal P/p^k\Lambda|)\]
cosets of $mQ^k\Lambda$ in $\Lambda$.  For each such coset $\bs q+mQ^k\Lambda$ we have
\begin{equation}
(\bs q+mQ^k\Lambda)\cap L(V_P;\mathcal P,\emptyset)=(\bs q+mQ^k\Lambda)\cap L(V_{P/Q};\mathcal P,\emptyset).
\end{equation}

Now write $\mathcal P=\{\bs p_1,\ldots,\bs p_r\}$. Since $\bs p_i+\bs t\in V_{P/Q}$
if and only if $c_k(\bs p_i+\bs t)$ is coprime to ${P/Q}$, it follows
from \eqref{Mobius} that for each of these cosets the cardinal of
$L(V_P;\mathcal P,\emptyset)\cap(\bs q+mQ^k\Lambda)\cap B_R(\bs x)$ is
\[\sum_{\substack{\bs t\in\Lambda\cap B_R(\bs x)\\
\bs t-\bs q\in mQ^k\Lambda}}
\hspace{1.5mm}\prod_{i=1}^r\hspace{-2mm}
\sum_{\substack{d\mid P/Q\\
d\mid c_k(\bs p_i+\bs t)}}
\mu(d).\]
Reversing the order of summation gives
\begin{equation}\label{VPcount}
\begin{array}[t]{c}
\displaystyle\sum_{d_1\mid P/Q}\hspace{1mm}\sum_{d_2\mid P/Q}\cdots\sum_{d_r\mid P/Q}\\
\text{\scriptsize$d_i^k<S/\lambda$ for each $i$}
\end{array}
\mu(d_1d_2\cdots d_r)\hspace{-4mm}
\sum_{\substack{\bs t\in\Lambda\cap B_R(\bs x)\\[.5mm]
\bs t\in\bs q+mQ^k\Lambda\\[1mm]
\bs t\in -\bs p_i+d_i^k\Lambda}}
1,
\end{equation}
where replacing $\mu(d_1)\cdots\mu(d_r)$ by $\mu(d_1\cdots d_r)$
is justified by the fact that the $d_i$'s are pairwise coprime since
any common factor of $c_k(\bs p_i+\bs t)$ and $c_k(\bs p_j+\bs t)$
divides $c_k(\bs p_i-\bs p_j)$, all of whose
prime factors divide $Q$.  Writing $d_1\cdots d_r=d$ and noting that
$(mQ,d)=1$, we can apply \eqref{LambdaCount} with $\Lambda$ replaced
by $m(dQ)^k\Lambda$ to obtain, for the inner sum, the estimate
\begin{equation}\label{innersum}
\frac{R^nv_n}{m^n(dQ)^{nk}}+O(R^{n-1}/m^{n-1}(dQ)^{(n-1)k})+O(1).
\end{equation}
Substituting this estimate in \eqref{VPcount} gives a main term
\[\frac{R^nv_n}{m^nQ^{nk}}\prod_{p\mid P/Q}\left(1-\frac{r}{p^{nk}}\right)\]
with error term \eqref{error}.  The main term arises by removing
the conditions $d_i^k<S/\lambda$ from the sum of the main term
in \eqref{innersum} then using the fact that $\mu(d)\tau_r(d)$
(where $\tau_r(d)$ is the number of ways of expressing $d$ as a
product of $r$ natural numbers) is a multiplicative function,
whose value is $-r$ at primes and 0 at prime powers, to express
the extended sum as an Euler product, as in \eqref{Eulerprod}.
The first error term in \eqref{error} comes from the extra terms
included in the extended multiple sum, so is
\[\le\frac{rR^nv_n}{m^nQ^{nk}}\sum_{d_1^k\ge S/\lambda}\frac{1}{d_1^{nk}}\sum_{d_2=1}^\infty\frac{1}{d_2^{nk}}
\cdots\sum_{d_r=1}^\infty\frac{1}{d_r^{nk}}=O\left(R^nS^{-n+(1/k)}\right),\]
since each of the $r-1$ complete sums is $\le\zeta(2)<2$.
The other two error terms account for the sum over the error terms in
\eqref{innersum}. The logarithms in the middle error term are necessary
only in the case $n=2$, $k=1$, when the series $\sum d_i^{-(n-1)k}$
diverge but the partial sums can be estimated by using \eqref{sum1/d}
or the standard estimate for the partial sums of the harmonic series.
In all other cases these series converge and the middle error term
can be taken as $O(R^{n-1})$.  (When $n=1$, there is no middle error
term, since the first error term in \eqref{innersum} is then the
same as the last.)

Finally, summing over the cosets of $mQ^k\Lambda$ that make up
$L(V_Q;\mathcal P,\emptyset)\cap(\bs m+m\Lambda)$ gives the main term \eqref{main} (since
$|\mathcal P/p^k\Lambda|=r$ when $p\nmid Q$) and increases the error
term by a factor at most $Q^k$, which is bounded in terms of $k$ and $\mathcal P$.
\end{proof}

\begin{corollary}\label{cor}
If $\rho$ is a positive radius and $P$ is a natural number divisible
by every prime less than $\log\rho$, then
\begin{equation}\label{asymptotP}
|V_P\cap B_\rho(\bs0)|=\frac{\rho^nv_n}{\zeta(nk)}+o(\rho^n),
\end{equation}
\begin{equation}\label{asymptot}
|V\cap B_\rho(\bs0)|=\frac{\rho^nv_n}{\zeta(nk)}+o(\rho^n),
\end{equation}
\begin{equation}\label{largeP}
|(V_P\setminus V)\cap B_\rho(\bs0)|=o(\rho^n)
\end{equation}
and, for any $\bs x\in\mathbb R^n$,
\begin{equation}\label{ubound}
|V_P\cap B_\rho(\bs x)|\le\frac{\rho^nv_n}{\zeta(nk)}+o(\rho^n),
\end{equation}
where $\zeta$ is the Riemann $\zeta$-function.
\end{corollary}

\begin{proof}
For \eqref{asymptotP} we use the lemma with $\mathcal P=\{\bs0\}$, $m=1$,
$\bs x=\bs0$ and $R=\rho$, then replace the product by $1/\zeta(nk)$
using \eqref{Eulertail}, with $N=\log\rho$, and \eqref{zetafn}.
This gives \eqref{asymptotP} with error term $O(\rho^n/\log^{nk-1}\rho)$.

Clearly $V_Q\cap B_\rho(\bs0)=V\cap B_\rho(\bs0)$ when $Q$ is the
product of all primes less than $(\rho/\lambda)^{1/k}$ (where
$\lambda$ is the length of the shortest nonzero vector in $\Lambda$),
giving \eqref{asymptot}, and \eqref{largeP} results from subtracting
this from \eqref{asymptotP}.

For \eqref{ubound} we use the lemma with $P$ replaced by $P'$, the
product of the primes less than $\log\rho$, together with \eqref{Eulertail}
and \eqref{zetafn}, and note that $V_P\subset V_{P'}$.  Then
\[|V_P\cap B_\rho({\bs x})|\le |V_{P'}\cap B_\rho({\bs x})|=
\frac{\rho^nv_n}{\zeta(nk)}+O\bigl(\rho^n/\log^{nk-1}\rho\bigr),\]
since $\log P'=O(\log\rho)$ and hence $\log\log P'$ and $\tau(P')$
are both $O(\rho^\epsilon)$ by \eqref{primeprod} and \eqref{taubound}.
\end{proof}

We note that \eqref{asymptot} tells us that $V$ has density
$1/\zeta(nk)$, generalizing Propositions~6 and 11 of \cite{BMP}
(though the error terms are not as good as those in \cite{BMP}
and much worse than those in \cite{Mir2} and \cite{Mir1}).
Also, one might regard \eqref{ubound} as saying that $V$ has a ``uniform
upper density" (or that $\Lambda\setminus V$ has a uniform lower density).

The following two theorems carry over to $k$-free points the results
of Mirsky \cite{Mir1,Mir2} (\cite{Mir1} improves
the error terms in \cite{Mir2})\footnote{As pointed out by J\"org
  Br\"udern, the work of Tsang~\cite{Tsang} can be extended to
  the case of $k$-free numbers and gives a further small improvement.} on $k$-free numbers.  A weaker result for squarefree
numbers goes back to Pillai \cite{P}.  Again, we make no attempt
in Theorem~\ref{density} to match the error term of \cite{Mir1},
postponing this to Section~\ref{errorterm}.

\begin{theorem}\label{density}
For any two disjoint finite subsets $\mathcal P$ and $\mathcal Q$ of $\Lambda$,
$L(V;\mathcal P,\mathcal Q)$ has a well defined density given by
\[\sum_{\mathcal F\subset\mathcal Q}(-1)^{|\mathcal F|}
\prod_p\left(1-\frac{|(\mathcal P\cup\mathcal F)/p^k\Lambda|}{p^{nk}}\right).\]
\end{theorem}

\begin{proof}
By the inclusion-exclusion principle \eqref{in-ex} applied to
$L(V;\mathcal P,\emptyset)\cap B_R(\bs0)$, with $P_i$ being the property that
$\bs q_i+\bs t\in V$ (where $\mathcal Q=\{\bs q_1,\bs q_2,\ldots\}$), we have
\[|L(V;\mathcal P,\mathcal Q)\cap B_R(\bs0)|=
\sum_{\mathcal F\subset\mathcal Q}(-1)^{|\mathcal F|}|L(V;\mathcal P\cup\mathcal F,\emptyset)\cap B_R(\bs0)|.\]
Now Lemma~\ref{VPpatches} with $P$ equal to the product of
the primes less than $\log R$ gives the estimate
\[R^nv_n\prod_{p<\log R}\biggl(1-\frac{|(\mathcal P\cup\mathcal F)/p^k\Lambda|}{p^{nk}}\biggr)
+O(R^{1/k}+R^{n-1+\epsilon})\]
for $|L(V_P;\mathcal P\cup\mathcal F,\emptyset)\cap B_R(\bs0)|$, the
proof of \eqref{largeP}
of Corollary~\ref{cor} shows that $V_P$ can be replaced by $V$ at
the expense of an extra error term $O(R^n/(\log R)^{nk-1})$, and
\eqref{Eulertail} allows the product to be extended over all primes
with a similar extra error term. Altogether, this gives the estimate
\[R^nv_n\sum_{\mathcal F\subset\mathcal Q}(-1)^{|\mathcal F|}\prod_{p}\biggl(1-\frac{|(\mathcal P\cup\mathcal F)/p^k\Lambda|}{p^{nk}}\biggr)
+O(R^n/(\log R)^{nk-1})\]
for $|L(V;\mathcal P,\mathcal Q)\cap B_R(\bs0)|$.
\end{proof}

\begin{theorem}\label{equivalence}
For disjoint finite subsets $\mathcal P$ and $\mathcal Q$
of $\Lambda$, the following statements are equivalent:
\begin{itemize}
\item[(i)]$|\mathcal P/p^k\Lambda|<p^{nk}$ for every prime $p$;
\item[(ii)]$L(V;\mathcal P,\mathcal Q)$ is non-empty;
\item[(iii)]$L(V;\mathcal P,\mathcal Q)$ has positive density.
\end{itemize}
\end{theorem}

\begin{proof}
Clearly (iii) implies (ii) and, almost as clearly, (ii) implies (i),
since if $\mathcal P$ contains a complete set of coset representatives
for $p^k\Lambda$ then, for every $\bs t\in\Lambda$, some point of
$\mathcal P+\bs t$ is in $p^k\Lambda$ (so not in $V$).

Now assume (i) holds.  For each $\bs q\in\mathcal Q$ choose a
different prime $p(\bs q)>(D(\mathcal P\cup\mathcal Q)/\lambda)^{1/k}$
and let $m$ be the product of the $p(\bs q)$'s.  By the Chinese
Remainder Theorem, there is an $\bs m\in\Lambda$ such that
\[\bs t\equiv\bs m\mbox{ (mod $m^k\Lambda$)}\iff
\bs t\equiv -\bs q\mbox{ (mod $p(\bs q)^k\Lambda$)}\quad\forall\bs q\in\mathcal Q.\]
Then for $\bs t\equiv\bs m$~(mod $m^k\Lambda$)
we have $\mathcal Q+\bs t\subset\Lambda\setminus V$ and
$\mathcal P+\bs t\subset V_m$ (the latter using the fact that
for every $\bs p\in\mathcal P$ and every prime factor $p(\bs q)$ of $m$,
$\bs q+\bs t\in p(\bs q)^k\Lambda$ and $\|\bs p-\bs q\|<p(\bs q)^k\lambda$,
ensuring that $\bs p+\bs t\not\in p(\bs q)^k\Lambda$).  Now Lemma~\ref{VPpatches},
with $P$ the product of the primes less than $\log R$ not dividing $m$,
gives a main term $CR^n$ with error $O(\max\{R^{n-1}\log R,R^{1/k}\})$
for the cardinal of a subset of the points
$\bs t\in\Lambda\cap B_R(\bs0)$ with $\mathcal P+\bs t\subset V_{mP}$
and $\mathcal Q+\bs t\subset\Lambda\setminus V$, where the constant $C$
is positive since the product in \eqref{main} has every term positive.
By \eqref{largeP} of Corollary~\ref{cor}, the number of these
points with $\mathcal P+\bs t\not\subset V$ is $o(R^n)$.  Hence
$L(V;\mathcal P,\mathcal Q)$ has positive lower density, and
so, by Theorem~\ref{density}, positive density.
\end{proof}

An interesting feature of Theorem~\ref{equivalence} is that the criterion
(i) is independent of $\mathcal Q$. This means, for example, that
\[L(V;\mathcal P,\emptyset)\ne\emptyset\Rightarrow
\delta(L(V;\mathcal P,\mathcal Q))>0\quad
\forall\mathcal Q\mbox{ with }\mathcal P\cap\mathcal Q=\emptyset,\]
which tells us, in particular, that every subset of a patch of $V$ is a patch of $V$.

\section{Examples}\label{examples}

Theorem~\ref{density} allows us to calculate the frequencies
of $\rho$-patches of $V$ in terms of the products
\[\Pi_r(nk):=\prod_{p>r^{1/nk}}\left(1-\frac{r}{p^{nk}}\right)\]
for $r=0,1,\ldots,|\Lambda\cap B_\rho(\bs0)|$.  Here, we
give two simple examples that both have $nk=2$ and that have
$|\Lambda\cap B_\rho(\bs0)|=3$ and 5, respectively.  So we
need the products
\[\begin{array}{l}
\Pi_1(2)=1/\zeta(2)=6/\pi^2=0.6079271\ldots,\\[1mm]
\Pi_2(2)=0.3226340\ldots\mbox{ (the Feller-Tornier constant),}\\[1mm]
\Pi_3(2)=0.1254869\ldots,\quad\Pi_4(2)=0.3785994\ldots,\quad\Pi_5(2)=0.2733455\ldots,
\end{array}\]
whose values can be calculated efficiently by the method described in \cite{Mor}.

Our first example is to find the frequencies of all 2-patches when $V$ is
the set of squarefree numbers.  Here $\Lambda=\mathbb Z$, $n=1$, $k=2$ and
$|\Lambda\cap B_2(\bs0)|=3$.  Since $-1,0,1$ are distinct mod $p^k$, for
every $p$, $|(\mathcal P\cup\mathcal F)/p^k|=|\mathcal P\cup\mathcal F|$
and $\nu(\mathcal P)$ depends only on $|\mathcal P|$ in this case.
Table~\ref{SFtable} gives the frequencies of 2-patches of all possible sizes,
both in terms of the above products and numerically, and Figure~\ref{SFfig}
depicts the patches themselves, with their frequencies. There are 3 patches
each of sizes 1 and 2, and we check that the sum, $\sum\nu(\mathcal P)$, of
the frequencies of all patches is 1 and that the average patch size,
$\sum\nu(\mathcal P)|\mathcal P|$, is $3\delta(V)=18/\pi^2$. The patches of
size 2 are the most frequent, as is to be expected since 2 is the closest
integer to $3\delta(V)$: indeed, 59\% of all locations have patches of size 2.
The empty patch is by far the rarest, occurring at less than 2\% of locations. The
radius $\rho=2$ is the largest for which every subset of $\Lambda\cap B_\rho(\bs0)$
occurs as a patch of $V$: of the 32 subsets of $\Lambda\cap B_3(\bs0)$ the 3
that contain 4 or 5 consecutive points do not occur as patches of $V$.

\begin{table}
\[\begin{array}{c|r|l}
|\mathcal P|&\multicolumn{1}{c|}{\nu(\mathcal P)}&
\multicolumn{1}{c}{\text{Numerical value}}\\
\hline
\rule{0pt}{5mm}3&\Pi_3(2)&0.125486980905\ldots\\
\rule{0pt}{5mm}2&\Pi_2(2)-\Pi_3(2)&0.197147118033\ldots\\
\rule{0pt}{5mm}1&\Pi_1(2)-2\Pi_2(2)+\Pi_3(2)&0.088145884881\ldots\\
\rule{0pt}{5mm}0&1-3\Pi_1(2)+3\Pi_2(2)-\Pi_3(2)&0.018634010349\ldots
\end{array}\]
\caption{Frequencies of the 2-patches of the squarefree numbers.}\label{SFtable}
\end{table}

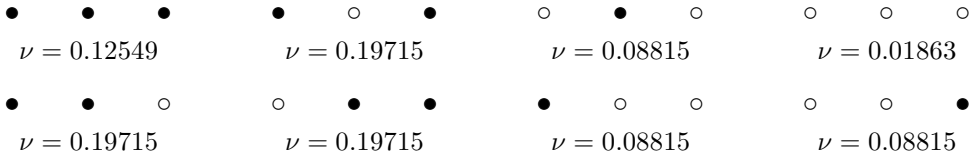
\begin{figure}
\centering
\begin{picture}(125,21)(0,-1)
\multiput(0,5)(35,0){4}{\multiput(0,0)(0,12){2}{\multiput(0,0)(10,0){3}{\circle{1.5}}}}
\multiput(0,17)(10,0){3}{\circle*{1.5}}
\multiput(35,17)(20,0){2}{\circle*{1.5}}
\put(80,17){\circle*{1.5}}
\multiput(0,5)(10,0){2}{\circle*{1.5}}
\multiput(45,5)(10,0){2}{\circle*{1.5}}
\put(70,5){\circle*{1.5}}
\put(125,5){\circle*{1.5}}
\put(10,12){\makebox(0,0){\small$\nu=0.12549$}}
\put(45,12){\makebox(0,0){\small$\nu=0.19715$}}
\put(80,12){\makebox(0,0){\small$\nu=0.08815$}}
\put(115,12){\makebox(0,0){\small$\nu=0.01863$}}
\put(10,0){\makebox(0,0){\small$\nu=0.19715$}}
\put(45,0){\makebox(0,0){\small$\nu=0.19715$}}
\put(80,0){\makebox(0,0){\small$\nu=0.08815$}}
\put(115,0){\makebox(0,0){\small$\nu=0.08815$}}
\end{picture}
\caption{The 2-patches of the squarefree numbers ($\Lambda=\mathbb Z$, $k=2$) with
their frequencies accurate to 5 decimal places. The black dots are points of $V$
and the open circles other lattice points.  The top row contains the patches with
mirror symmetry and the bottom row the two mirror image pairs.}\label{SFfig}
\end{figure}

Our other example is the $\sqrt2$-patches of the visible points, $V$, in $\mathbb Z^2$,
where $\Lambda=\mathbb Z^2$, $n=2$, $k=1$ and $|\Lambda\cap B_2(\bs0)|=5$.
Figure~\ref{VISfig} shows the different patches, up to symmetry, with their frequencies.
The four patches in the top row have the full dihedral symmetry $D_4$; the two in the
second row have symmetry $D_2$, and give rise to another patch on rotation through
$\pi/2$; the remaining six patches have only reflection symmetry, and each gives rise
to three others on rotation through $\pm\pi/2$ and $\pi$. We can again check that
$\sum\nu(\mathcal P)=1$ and $\sum\nu(\mathcal P)|\mathcal P|=5\delta(V)=30/\pi^2$. This
time, however, the frequencies do not depend only on $|\mathcal P|$, and indeed no two
patches that are not symmetry related have the same frequency. Of the five patches
with $|\mathcal P|=4$, the symmetric one has frequency nearly 5 times that of each
of the other four, and the ratio of the frequencies of two of the patches with
$|\mathcal P|=3$ is nearly 30. Surprisingly, one of the patches with $|\mathcal P|=3$
(the commonest patch size) has frequency smaller than that of any patch except the
empty one. The empty patch itself occurs at less than 1 in 900 locations.  As in
the previous example, $\sqrt2$ is the largest radius for which every subset of
$\Lambda\cap B_\rho(\bs0)$ is a patch: of the 512 subsets of $\Lambda\cap B_{\sqrt3}(\bs0)$,
the 135 that contain all four vertices of a lattice square do not occur as patches of $V$.

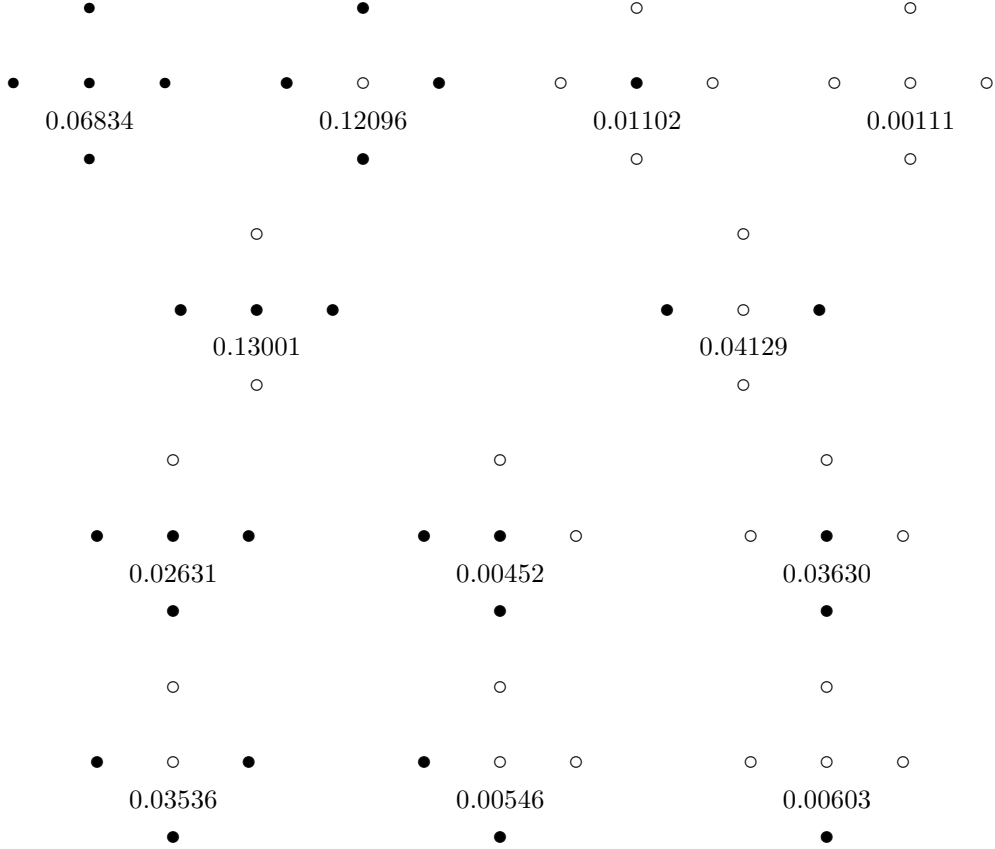
\begin{figure}
\centering
\begin{picture}(128,110)(0,5)
\newcommand{\patch}{\multiput(-10,0)(10,0){3}{\circle{1.5}}\multiput(0,-10)(0,20){2}{\circle{1.5}}}
\multiput(46,105)(36,0){3}\patch
\multiput(32,75)(64,0){2}\patch
\multiput(0,0)(0,30){2}{\multiput(21,15)(43,0){3}\patch}
\multiput(0,105)(10,0){3}{\circle*{1.5}}\multiput(10,95)(0,20){2}{\circle*{1.5}}
\multiput(36,105)(20,0){2}{\circle*{1.5}}\multiput(46,95)(0,20){2}{\circle*{1.5}}
\put(82,105){\circle*{1.5}}
\multiput(22,75)(10,0){3}{\circle*{1.5}}
\multiput(86,75)(20,0){2}{\circle*{1.5}}
\multiput(21,45)(43,0){3}{\circle*{1.5}}
\multiput(31,15)(0,30){2}{\circle*{1.5}}
\multiput(0,0)(0,30){2}{\multiput(11,15)(43,0){2}{\multiput(0,0)(10,-10){2}{\circle*{1.5}}}}
\multiput(107,5)(0,30){2}{\circle*{1.5}}
\put(10,100){\makebox(0,0){\small0.06834}}
\put(46,100){\makebox(0,0){\small0.12096}}
\put(82,100){\makebox(0,0){\small0.01102}}
\put(118,100){\makebox(0,0){\small0.00111}}
\put(32,70){\makebox(0,0){\small0.13001}}
\put(96,70){\makebox(0,0){\small0.04129}}
\put(21,40){\makebox(0,0){\small0.02631}}
\put(64,40){\makebox(0,0){\small0.00452}}
\put(107,40){\makebox(0,0){\small0.03630}}
\put(21,10){\makebox(0,0){\small0.03536}}
\put(64,10){\makebox(0,0){\small0.00546}}
\put(107,10){\makebox(0,0){\small0.00603}}
\end{picture}
\caption{The $\sqrt2$-patches of the visible points of $\mathbb Z^2$, up to
symmetry, with their frequencies, accurate to 5 decimal places.}\label{VISfig}
\end{figure}

\section{Entropy calculations}\label{entropy}

\begin{theorem}\label{visible}
$h_\T(V)=1/\zeta(nk)$.
\end{theorem}

\begin{proof}
For each radius $\rho>0$ let $P=P(\rho)$ be the product of the primes
$p$ with $p^{nk}\le|\Lambda\cap B_\rho(\bs0)|$. Then $P$ is divisible
by every prime less than $\log\rho$ when $\rho$ is large enough.

The $\rho$-patch of $V$ at any point $\bs t\in\Lambda$ is a subset of
$(V_P-\bs t)\cap B_\rho(\bs0)$ and by \eqref{ubound} of Corollary~\ref{cor}
the cardinal of this set is at most $\rho^nv_n/\zeta(nk)+o(\rho^n)$.
Also there are at most $P^{nk}$ possibilities for $V_P-\bs t$ as $\bs t$
varies, since $P^k\Lambda$ is the lattice of periods of $V_P$.  So
\begin{eqnarray}\label{upperbd}
\log_2N(\rho)&\le&\frac{\rho^nv_n}{\zeta(nk)}+o(\rho^n)+nk\log_2P\nonumber\\
&\le&\frac{\rho^nv_n}{\zeta(nk)}+o(\rho^n)+O(\rho^{1/k}),
\end{eqnarray}
since $\log_2P=O(\rho^{1/k})$, by \eqref{primeprod}.

To bound $\log_2N(\rho)$ below we note that every subset $\mathcal P$ of
$V\cap B_\rho(\bs0)$ is the \mbox{$\rho$-patch} of $V$ at
some point of $\Lambda$, by Theorem~\ref{equivalence} with
$\mathcal Q=\Lambda\cap B_\rho(\bs0)\setminus\mathcal P$.
By \eqref{asymptot} of Corollary~\ref{cor},
$|V\cap B_\rho(\bs0)|=\rho^nv_n/\zeta(nk)+o(\rho^n)$, so
\begin{equation}\label{lowerbd}
\log_2N(\rho)\ge\frac{\rho^nv_n}{\zeta(nk)}+o(\rho^n).
\end{equation}
On dividing by $\rho^nv_n$ and letting $\rho$ tend to infinity,
\eqref{upperbd} and \eqref{lowerbd} give $h_\T(V)=1/\zeta(nk)$.
\end{proof}

To bound the measure entropy we need the following lemma, which
enables us to obtain good upper bounds for the frequency of ``sparse"
patches of $V$, i.e.\ patches that contain few points in comparison to their size.

\begin{lemma}\label{nubound}
Let $\mathcal P$ and $\mathcal Q$ be disjoint finite subsets of $\Lambda$,
let $Q$ be the product of all primes $p$ with
\begin{equation}\label{Q}
|\mathcal Q/p^k\Lambda|<|\mathcal Q|
\end{equation}
and define
\[s:=\min_{\bs t\in L(V;\mathcal P,\mathcal Q)}|(\mathcal Q+\bs t)\cap V_Q|.\]
Then
\begin{equation}\label{udensity}
\delta(L(V;\mathcal P,\mathcal Q))=
O\bigl(4^{(D(\mathcal Q)/\lambda)^{1/k}nk}/|\mathcal Q|^{s-s/nk}\bigr),
\end{equation}
where the $O$-constant depends only on $\Lambda$.
\end{lemma}

\begin{proof}
If $\bs t\in L(V;\mathcal P,\mathcal Q)$ then for each
$\bs q\in\mathcal Q\cap(V_Q-\bs t)$ there is a prime
$p(\bs q)\nmid Q$ with $\bs q+\bs t\in p(\bs q)^k\Lambda$,
and by the definition of $Q$ these primes are distinct.
By \eqref{largeP} of Corollary~\ref{cor} with
$\rho=R+\max_{\bs q\in\mathcal Q}\|\bs q\|$ and $P$ the product
of the primes less than $\log\rho$, the number of points
$\bs t\in  L(V;\mathcal P,\mathcal Q)\cap B_R(\bs0)$ for which
$p(\bs q)\ge\log\rho$ for some $\bs q\in\mathcal Q\cap(V_Q-\bs t)$ is $o(\rho^n)$.
The remaining $\bs t$'s in $ L(V;\mathcal P,\mathcal Q)\cap B_R(\bs0)$
have $p(\bs q)<\log\rho$ for each $\bs q\in\mathcal Q\cap(V_Q-\bs t)$.
For the number of such $\bs t$ with a given set
$\mathcal Q\cap(V_Q-\bs t)=\{\bs q_1,\ldots,\bs q_t\}$ and a
given ordered set of primes $\{p(\bs q_1),\ldots,p(\bs q_t)\}$,
\eqref{LambdaCount} with $\Lambda$ replaced by
$(p(\bs q_1)\cdots p(\bs q_t))^k\Lambda$, gives the estimate
\[\le\frac{CR^n}{(p(\bs q_1)\cdots p(\bs q_t))^{nk}}\]
when $R$ is large enough to ensure that $R^n>(\log\rho)^{nk|\mathcal Q|}$,
where the constant $C$ depends only on $\Lambda$. The sum of
this over all sets of $t$ primes not dividing $Q$ is majorized by
\[CR^n\Biggl(\sum_{p\nmid Q}\frac{1}{p^{nk}}\Biggr)^t
<CR^n\Biggl(\sum_{m\ge|\mathcal Q|^{1/nk}}\frac{1}{m^{nk}}\Biggr)^t
<\frac{CR^n}{|\mathcal Q|^{s-s/nk}},\]
since $t\ge s$ and the least prime not dividing $Q$ is
$\ge|\mathcal Q|^{1/nk}$, by \eqref{Q}.  There are at most
$Q^{nk}$ possibilities for $\mathcal Q\cap(V_Q-\bs t)$,
since $Q^k\Lambda$ is the lattice of periods of $V_Q$, so
\begin{eqnarray*}
|L(V;\mathcal P,\mathcal Q)\cap B_R(\bs0)|&<&
\frac{Q^{nk}CR^n}{|\mathcal Q|^{s-s/nk}}+o(R^n)\\
&<&\frac{4^{(D(\mathcal Q)/\lambda)^{1/k}nk}CR^n}{|\mathcal Q|^{s-s/nk}}+o(R^n)
\end{eqnarray*}
for large $R$, where the second inequality results from
\eqref{primeprod} and the fact that $p^k\le D(\mathcal Q)/\lambda$
for every prime factor $p$ of $Q$.  The result follows on dividing
by $R^n$ and letting $R$ tend to infinity (the existence of the
limit on the left being guaranteed by Theorem~\ref{density}).
\end{proof}

\begin{theorem}
$h_\M(V)=0$.
\end{theorem}

\begin{proof}
Given $\rho>0$ and a $\rho$-patch $\mathcal P$ of $V$,
let $\mathcal Q:=(\Lambda\cap B_\rho(\bs 0))\setminus\mathcal P$
and, as in Lemma~\ref{nubound}, define $Q=Q(\mathcal P)$ to be the
product of all primes $p$ with $|\mathcal Q/p^k\Lambda|<|\mathcal Q|$ and
\[s=s(\mathcal P):=\min_{\bs t\in L(V;\mathcal P,\mathcal Q)}|(\mathcal Q+\bs t)\cap V_Q|.\]
By \eqref{ubound} of Corollary~\ref{cor} and the fact that
$V\subset V_P$ with $P$ the product of primes less than $\log\rho$, we have
\begin{equation}\label{setQ}
|\mathcal Q|>
\left(1-\frac{1}{\zeta(nk)}\right)v_n\rho^n-o(\rho^n)>\frac{v_n\rho^n}{2^{nk}}
\end{equation}
for large enough $\rho$.

Now put $S=S(\rho):=\rho^n/\sqrt{\log_2\rho}$.
We shall calculate separately the contributions to the measure entropy $h_\M$
of the $\rho$-patches $\mathcal P$ of $V$ with $s(\mathcal P)\ge S$ and those with $s(\mathcal P)< S$.
The former patches have small frequency and the latter are few in number.

For the $\rho$-patches with $s(\mathcal P)\ge S$, Lemma~\ref{nubound}
and \eqref{setQ} give
\[-\log_2\nu(\mathcal P)>\frac{n}{2}S\log_2\rho-O\bigl(D(\mathcal Q)^{1/k}\bigr)
=\frac{n}{2}S\log_2\rho-O\bigl(\rho^{1/k}\bigr)>\frac{S\log_2\rho}{3}\]
for large enough $\rho$ which, since $-\log_2\nu$ is decreasing but
$-\nu\log_2\nu$ is increasing for $\nu\in(0,1/e]$, gives the estimate
\[-\nu(\mathcal P)\log_2\nu(\mathcal P)=O(2^{-S\log_2\rho/3}S\log_2\rho).\]
Since there are at most $2^{|\Lambda\cap B_\rho(\bs0)|}$ $\rho$-patches
in all, the contribution of the $\rho$-patches $\mathcal P$ with $s(\mathcal P)\ge S$
to the sum on the right of \eqref{met} is
\begin{equation}
O\bigl(2^{|\Lambda\cap B_\rho(\bs0)|-(S\log_2\rho)/3}S\log_2\rho\bigr)=
O\Bigl(2^{-\frac{\rho^n}{4}\sqrt{\log_2\rho}}\rho^n\sqrt{\log_2\rho}\Bigr)
=o(1).\label{larges}
\end{equation}

Turning to the $\rho$-patches with $s(\mathcal P)<S$, denote this set of patches
by $\mathcal B\subset\mathcal A(\rho)$ and let $F$ be their combined frequency.
The contribution of these patches to the sum on the right of \eqref{met} is
\[\sum_{\mathcal P\in\mathcal B}-\nu(\mathcal P)\log_2\nu(\mathcal P)\]
which, since $\nu\log_2\nu$ is a convex function of $\nu$, does not decrease
if we replace the $\nu(\mathcal P)$'s by their average value, $F/|\mathcal B|$.
So this contribution is
\begin{equation}\label{Bbound}
\le F\log_2|\mathcal B|-F\log_2F\le\log_2|\mathcal B|+\frac{\log_2e}{e}.
\end{equation}
To bound $|\mathcal B|$ we note that if $\mathcal P\in\mathcal B$
then there is a $\bs t\in L(V;\mathcal P,\mathcal Q)$ with $|\mathcal Q\cap(V_Q-\bs t)|<S$.
Since $\mathcal P\subset V-\bs t\subset V_Q-\bs t$,
$\mathcal Q=(\Lambda\cap B_\rho(\bs0))\setminus\mathcal P$, and
$Q^k\Lambda$ is the lattice of periods of $V_Q$, $\mathcal P$
and $\mathcal Q$ are completely determined by this subset of
$B_\rho(\bs0)$ and by $\bs t$ modulo $Q^k\Lambda$.  There are $Q^{nk}$
cosets of $Q^k\Lambda$ in $\Lambda$ and the number of subsets of
$\Lambda\cap B_\rho(\bs0)$ with fewer than $S$ members is bounded above by
\[\sum_{i=0}^{\lfloor S\rfloor}\binom{|\Lambda\cap B_\rho(\bs0)|}{i}
\le(\lfloor S\rfloor+1)\binom{|\Lambda\cap B_\rho(\bs0)|}{\lfloor S\rfloor}
\le2S\left(\frac{e|\Lambda\cap B_\rho(\bs0)|}{S}\right)^{\!\!S}\]
for large $\rho$, by \eqref{binomial}. Hence the bound
on the right of \eqref{Bbound} is majorized by
\begin{eqnarray}
\lefteqn{S\log_2(e|\Lambda\cap B_\rho(\bs0)|/S)+\log_22eS+nk\log_2Q}\nonumber\\[1mm]
&&=O\left(\frac{\rho^n\log_2\log_2\rho}{\sqrt{\log_2\rho}}\right)+O(\log_2\rho)+O(\rho^{1/k})
=o(\rho^n).\label{smalls}
\end{eqnarray}

Since the contributions \eqref{larges} and \eqref{smalls} are both
$o(\rho^n)$, $h_\M(V)=0$.
\end{proof}

\noindent{\bf Note on patch shapes.}\quad
On the general principle of the isotropy of space, we have used
spherical patches throughout and measured densities and frequencies
through expanding spherical regions; but the results we obtain are
independent of the shapes of these patches and regions: all our
point-counting estimates stem from \eqref{LambdaCount} which remains
valid for an arbitrary expanding region in place of the expanding
ball, with main term the volume of the region (using the volume
of the fundamental region of the lattice as a unit) and an error
term of smaller order provided the boundary of the region has
\mbox{$n$-dimensional} measure zero.  It is not even necessary
for the shape of the density-defining regions to be the same as
the (also expanding) patch shape.

\section{Variational principle}\label{variational}

Endowing the power set $\{0,1\}^{\Lambda}$ of the lattice $\Lambda$ with the
product topology of the discrete topology on $\{0,1\}$, it becomes a
compact topological space (by Tychonov's theorem). This
topology is in fact generated  by the metric $d$ defined by
$$
d(X,Y):=\min\left\{1,\inf\{\epsilon >0\,\mid\, X\cap
  B_{1/\epsilon}(\bs 0)=Y\cap
  B_{1/\epsilon}(\bs 0)\}\right\}
$$
for subsets $X,Y$ of $\Lambda$. Then $(\{0,1\}^{\Lambda},\Lambda)$ is
a \emph{topological dynamical system}, i.e.\ the natural translational
action of the group $\Lambda$
on $\{0,1\}^{\Lambda}$ is continuous. 

Now let $X$ be a subset of
$\Lambda$. The closure $\mathbb X(X)$ of the set of 
lattice translations $\bs t+X$ ($\bs t\in\Lambda$) of $X$ in
$\{0,1\}^{\Lambda}$ gives rise to the topological dynamical system
$(\mathbb X(X),\Lambda)$, i.e.\ $\mathbb X(X)$ is a compact topological space on which
the action of $\Lambda$ is continuous; cf.~\cite{BLR} and references therein for
details. Denote by $\mathcal{M}(\mathbb X(X),\Lambda)$ the set of $\Lambda$-invariant
probability measures on $\mathbb X(X)$ with respect to the Borel $\sigma$-algebra on $\mathbb X(X)$, i.e.\ the smallest 
$\sigma$-algebra on $\mathbb X(X)$ which contains the open subsets of
$\mathbb X(X)$. For a fixed such measure
$\mu$ and a radius $\rho>0$, let $h_\rho(\mu)$ be the entropy of $\mu$
restricted to $\mathcal{A}(\rho)$, i.e.\ 
$$h_\rho(\mu):=\sum_{\mathcal{P}\in\mathcal{A}(\rho)}-\mu(C_\mathcal{P})\log_2
\mu(C_\mathcal{P}),$$
where $\mathcal A(\rho)$ denotes the set of $\rho$-patches of $X$ and
$C_\mathcal{P}$ is the set of elements of $\mathbb X(X)$ whose $\rho$-patch at
$\bs 0$ is $\mathcal P$, the so-called cylinder set with respect
to $\mathcal P$. The \emph{metric entropy} of $\mu$ is then given by the limit 
$$h(\mu):=\lim_{\rho\to\infty}\frac{h_\rho(\mu)}{\rho^nv_n}$$
which exists by a subadditivity argument; cf.~\cite{BS} and also
see~\cite{ELW,Keller,Walters}. As in Section~\ref{defns}, replacing the $\mu(C_\mathcal{P})$'s by
their average value, $1/N(\rho)$, we see that
\[
h(\mu)\le h_\T(X)\quad
\forall \mu\in\mathcal M(\mathbb X(X),\Lambda).
\]
Since the
topological entropy $h_{\rm top}(\mathbb X(X),\Lambda)$ of $(\mathbb X(X),\Lambda)$
coincides with  $h_\T(X)$ by~\cite[Theorem 1 and Remark 2]{BLR}, the variational
principle for lattice actions on compact spaces here reads as follows; cf.~\cite{BS}
and~\cite[Sect.~6]{Ruelle}, the latter being an
extension of the case $n=1$ from~\cite{DGS,Walters}. An elementary
proof can be found in~\cite{Mis}. Note that the additional statement
follows from the expansiveness of the action of $\Lambda$ on $\mathbb X(X)$.

\begin{theorem}[Variational principle] 
$$
\sup_{\mu\in \mathcal{M}(\mathbb X(X),\Lambda)} h(\mu)=h_\T(X).
$$
Moreover, the supremum is achieved at some measure.
\qed
\end{theorem}

In case of $V$, $\mathbb X(V)$ will also contain the empty set
(cf.~Proposition~\ref{holes}) and various other subsets of
$\Lambda$ and thus admits
many $\Lambda$-invariant probability measures. In fact, we shall now
show that $\mathbb X(V)$ coincides with the set of \emph{admissible} subsets $A$
of $\Lambda$, i.e.\ subsets $A$ of $\Lambda$ having the property that
every finite subset $\mathcal P$ of $A$ satisfies criterion (i) of
Theorem~\ref{equivalence}; compare~\cite[Theorem 8(i)]{Sarnak}. We
denote the set of all admissible subsets of
$\Lambda$ by
$\mathbb A$.

\begin{theorem}
$\mathbb X(V)=\mathbb A$.
\end{theorem}
\begin{proof}
Since $V\in\mathbb A$ (otherwise some point of $V$ is in $p^k\Lambda$
for some prime $p$, a contradiction) and since $\mathbb A$ is a $\Lambda$-invariant and
closed subset of $\{0,1\}^{\Lambda}$, it follows that $\mathbb A$ contains $\mathbb X(V)$. For the
other inclusion, let $A\in\mathbb A$. Then, for any $\rho >0$,
Theorem~\ref{equivalence} applied to the finite subset
$A_\rho=A\cap B_\rho(\bs 0)$ of $A$ implies the existence of a $t_\rho\in L(V;A_\rho,\Lambda\cap B_\rho(\bs 0)\setminus A_\rho)$. It follows that $A\in\mathbb X(V)$.
\end{proof}

  Moreover, one has 
$h_\T(V)=1/\zeta(nk)$ by Theorem~\ref{visible}. Consider the frequency function $\nu$ from above which gives the
frequencies $\nu(\mathcal P)$ of occurence of $\rho$-patches $\mathcal
P$  
 of $V$ in space. The function $\nu$, regarded as a function on the
cylinder sets by setting $\nu(C_\mathcal
P):=\nu(\mathcal P)$, is finitely additive on the
cylinder sets with $\nu(\mathbb X(V))=\sum_{\mathcal P\in\mathcal
  A(\rho)}\nu(C_{\mathcal P})=1$. Since the
family of cylinder sets is a (countable) semi-algebra that generates the
Borel $\sigma$-algebra on $\mathbb X(V)$, one can use the method
from~\cite[\S 0.2]{Walters} to show that $\nu$ extends uniquely to a probability measure on
$\mathbb X(V)$. Moreover, this probability measure can be
seen to be $\Lambda$-invariant. This shows that the
measure entropy $h_\M(V)$ is indeed a metric entropy
of a $\Lambda$-invariant probability measure on $\mathbb X(V)$. Certainly, an explicit characterisation of $\mathcal M(\mathbb X(V),\Lambda)$
together with the corresponding metric entropies (in particular
those measures $\mu\in \mathcal{M}(\mathbb X(V),\Lambda)$ with maximal entropy,
i.e.\ $h(\mu)=1/\zeta(nk)$) would be desirable (but not simple).

\section{Diffraction spectrum}\label{diffraction}

In the following, we assume that the reader is acquainted with the
mathematics of diffraction as carefully laid out in~\cite{BMP}; see
also~\cite{BG0} and references therein for a review. We
shall 
also use the notation and results from that text. In fact, the proofs
presented below are straightforward modifications of the corresponding
proofs in~\cite{BMP} and are only included for the reader's convenience. For an alternative derivation of the diffraction
spectrum in case of the visible lattice points,
see~\cite[Sect.~5a]{Sing}. 

A Dirichlet series we shall encounter below is
\begin{equation}\label{xi}
\xi(s):=\sum_{m=1}^{\infty}\frac{\mu(m)\tau(m)}{m^s}=\prod_p\left(1-\frac{2}{p^s}\right)\,,
\end{equation}
which is absolutely convergent for $\Re(s)>1$, where $\tau$
is the ordinary divisor function in \eqref{taubound}.

\begin{theorem}\label{auto}
The natural autocorrelation of $V$ exists and is supported on
$\Lambda$, the weight of a point $\bs a\in\Lambda$ in the
autocorrelation of $V$ being given by
$$
w(\bs a)=\xi(nk)\prod_{p\mid c_k(\bs a)}\left(1+\frac{1}{p^{nk}-2}\right)\,,
$$
with error term equal to $O(R^{-(1-(1/k))^2})$ for $n=1$ and $k\geq
2$, $O(R^{-1/2})$ for $n=2$ and $k=1$ and $O(R^{-1})$ otherwise,
where, in any case, the implied constant depends on $\bs
a$ as well as on $\Lambda$. (For lattices $\Lambda$ with determinant $\neq 1$ the weights above must be divided by\/ $\det(\Lambda)$.)
\end{theorem}
\begin{proof}
Since the cases $n=1$, $k\geq
2$ and $n\geq 2$, $k=1$ were already treated in~\cite[Theorems 1, 2 and
4]{BMP}, we may assume that $n,k\geq 2$. Since $V-V\subset\Lambda$, the autocorrelation of $V$ (if it exists)
can only be supported on $\Lambda$. The weight of a point $\bs a\in\Lambda$ in
the autocorrelation of $V$ is the limit as $R\to\infty$
of
\begin{equation}\label{weights}
\frac{1}{R^nv_n}\sum_{\bs x,\bs x-\bs a\in V\cap B_R(\bs 0)} 1
\end{equation}
and, by~\cite[Lemma 1]{BMP}, the existence of this limit for each
$\bs a \in\Lambda$ is sufficient to ensure the existence of the
autocorrelation.

It is convenient to drop the condition $\bs x-\bs a\in B_R(\bs 0)$ in~\eqref{weights}, which then becomes
\begin{equation}\label{weightsmod}
\frac{1}{R^nv_n}\sum_{\substack{\bs x,\bs
    x-\bs a\in V\\[.5mm]\bs x\in B_R(\bs 0)}} 1\,.
\end{equation}
The difference between these sums is $O(1/R)$ by~\eqref{LambdaCount}, due to the extra
lattice points $\bs x$ within a constant distance $\|\bs a\|$ of
the boundary of $B_R(\bs 0)$ that are included in the
latter. By~\eqref{Mobius}, this can be written as
\[
\frac{1}{R^nv_n}\sum_{\bs x\in\Lambda\cap B_R(\bs 0)\setminus\{\bs 0,\bs
    a\}}\hspace{1mm}\sum_{l\mid c_k(\bs
  x)}\mu(l)\hspace{1mm}\sum_{m\mid c_k(\bs x-\bs a)}\mu(m)\,.
\]
Reversing the order of summation gives
\[\frac{1}{R^nv_n}
\sum_{1\leq l<S^{\frac{1}{k}}}\sum_{1\leq
  m<S^{\frac{1}{k}}}\mu(l)\mu(m)\sum_{\substack{\bs x\in\Lambda\cap
    B_R(\bs 0)\setminus\{\bs 0,\bs
    a\}\\[.5mm]\bs x\in l^k\Lambda\\[.5mm]\bs
x-\bs a\in m^k\Lambda}}1\,,
\]
where $S:=(R+\|\bs a\|)/\lambda$. Collecting terms with the
same value of $d=(l,m)$, noting that all $\bs x$ in the inmost sum
belong to $d^k\Lambda$ and that there is no such $\bs x$ unless $\bs
a\in d^k\Lambda$, and putting $l':=l/d$, $m':=m/d$, $\bs x':=\bs x/d^k$,
$\bs a':=\bs a/d^k$, we obtain
\begin{equation}\label{weightsmod2}
\frac{1}{R^nv_n}
\sum_{d\mid c_k(\bs a)}\sum_{1\leq l'<S^{\frac{1}{k}}/d}\hspace{1mm}\sum_{\substack{1\leq
  m'<S^{\frac{1}{k}}/d\\[.5mm](l',m')=1}}\mu(l'd)\mu(m'd)\sum_{\substack{\bs x'\in\Lambda\cap
    B_{R/d^k}(\bs 0)\setminus\{\bs 0,\bs
    a'\}\\[.5mm]\bs x'\in l'^k\Lambda\\[.5mm]\bs
x'-\bs a'\in m'^k\Lambda}}1\,.
\end{equation}
Since $l'$ and $m'$ are bound variables of summation and $\bs x'$ and
$\bs a'$ will not be referred to again, we can drop the dashes: from
now on $l$ and $m$ are the new $l'$ and $m'$ but $\bs a$ is the
original $\bs a$.

By~\eqref{LambdaCount} with $\Lambda$ replaced by $(lm)^k\Lambda$, the
inmost sum is
\[
v_n\left(\frac{R}{(dlm)^k}\right)^n+O\left(\frac{R}{(dlm)^k}\right)^{n-1}+O(1)\,.
\]
These three terms give a main term and two error terms
in~\eqref{weightsmod2}. 

The first error term is majorized by 
\[
O\left(\frac{1}{R}\sum_{1\leq l<S^{\frac{1}{k}}}\frac{1}{l^{k (n-1)}}\sum_{1\leq m<S^{\frac{1}{k}}}\frac{1}{m^{k(n-1)}}\right)=O(1/R)
\]
since the sums are convergent due to $k(n-1)\geq 2$.

The second error term is majorized by
\[
O\left(\frac{S^{2/k}}{R^n}\right)=O\left(\frac{1}{R^{n-2/k}}\right)=O(1/R)
\] 
since $S=O(R)$ and $k(n-1)\geq 2$. So both error terms are $O(1/R)$
and thus tend to $0$ as $R\to\infty$.

The main term is
\begin{eqnarray*}
&&\sum_{d\mid c_k(\bs a)}\sum_{1\leq l<S^{\frac{1}{k}}/d}\hspace{1mm}\sum_{\substack{1\leq
  m<S^{\frac{1}{k}}/d\\[.5mm](l,m)=1}}\frac{\mu(ld)\mu(md)}{(dlm)^{nk}}\\&=&\sum_{d\mid c_k(\bs a)}\sum_{\substack{1\leq l<S^{\frac{1}{k}}/d\\[.5mm](l,d)=1}}\hspace{1mm}\sum_{\substack{1\leq
  m<S^{\frac{1}{k}}/d\\[.5mm](m,d)=1\\[.5mm](l,m)=1}}\frac{\mu(ld)\mu(md)}{(dlm)^{nk}}\\
&=&
\sum_{d\mid c_k(\bs a)}\frac{\mu^2(d)}{d^{nk}}\sum_{\substack{1\leq l<S^{\frac{1}{k}}/d\\[.5mm](l,d)=1}}\hspace{1mm}\sum_{\substack{1\leq
  m<S^{\frac{1}{k}}/d\\[.5mm](m,d)=1}}\frac{\mu(lm)}{(lm)^{nk}}
\end{eqnarray*}
since $\mu$ is multiplicative and $\mu(lm)=0$ when $(l,m)\neq 1$. Since the last double sum is absolutely convergent, this converges to
\begin{equation}\label{weightsmain}
\sum_{d\mid c_k(\bs a)}\frac{\mu^2(d)}{d^{nk}}\sum_{\substack{r=1\\[.5mm](r,d)=1}}^{\infty}\frac{\mu(r)\tau(r)}{r^{nk}}
\end{equation}
as $R\to \infty$. The difference between this limit and the partial
sum above is $O(1/R^{nk-1})$, so falls within the error estimate
$O(1/R)$. 

Using~\eqref{xi}, the expression for the limit~\eqref{weightsmain} can
be rearranged as
\begin{eqnarray*}
&&\sum_{\substack{d\mid c_k(\bs a)\\[.5mm]\text{\scriptsize $d$
      squarefree}}} \frac{1}{d^{nk}}\prod_{p\nmid
  d}\left(1-\frac{2}{p^{nk}}\right)\\
&=&\xi(nk)\sum_{\substack{d\mid c_k(\bs a)\\[.5mm]\text{\scriptsize $d$
      squarefree}}} \frac{1}{d^{nk}}\prod_{p\mid d}\left
  (1-\frac{2}{p^{nk}}\right )^{-1} \\
&=&\xi(nk)\prod_{p\mid c_k(\bs a)}\left(1+\frac{1}{p^{nk}}\left(1-\frac{2}{p^{nk}}\right)^{-1}\right)\\
&=&\xi(nk)\prod_{p\mid c_k(\bs a)}\left(1+\frac{1}{p^{nk}-2}\right)\,.
\end{eqnarray*}
This completes the proof.
\end{proof}

\begin{corollary}
$V-V=\Lambda$.
\end{corollary}
\begin{proof}
Trivially, one has $V-V\subset\Lambda$. From Theorem~\ref{auto}, one
gets $w(\bs a)>0$ for all $\bs a\in\Lambda$ and thus also
$\Lambda\subset V-V$ by~\cite[Lemma 1]{BMP}.
\end{proof}

The \emph{dual}\/ or \emph{reciprocal 
lattice}\/ $\Lambda^*$ of $\Lambda$ is
\[ 
\Lambda^*:=\{\bs y \in\mathbb{R}^n\mid \bs y\cdot\bs x\in\mathbb Z
\mbox{ for all } \bs x\in\Lambda\}
\]
By definition, the  \emph{denominator}\/ $q$ of a point $\bs p\in\mathbb
Q\Lambda^*$ is the smallest number $a\in\mathbb N$
with $a\bs p\in\Lambda^*$. This is also the greatest common divisor of the
numbers $a\in\mathbb N$ with $a\bs p\in\Lambda^*$, i.e.\ $a\bs
p\in\Lambda^*$ if and only if $q\mid a$.

\begin{theorem}
The diffraction measure $\widehat{\gamma}$ of the autocorrelation
$\gamma$ of\/ $V$ exists and is a positive, pure-point, 
translation-bounded measure
which is concentrated on the set of points in $\mathbb Q\Lambda^*$
with $(k+1)$-free denominator and whose intensity at a point with such
a denominator $q$ is given by
\begin{equation}\label{intensity}
\frac{1}{\zeta^2(nk)}\prod_{p\mid q}\frac{1}{(p^{nk}-1)^2}\,.
\end{equation}
This measure can also be interpreted as
\begin{equation}\label{diff}
\widehat{\gamma}=\xi(nk)\sum_{\substack{d=1\\[.5mm]\text{\scriptsize $d$
      squarefree}}}^{\infty}\left(\prod_{p\mid d}\frac{1}{p^{2nk}-2p^{nk}}\right)\omega_{\Lambda^*/d^k}\,,
\end{equation}
a weak*-convergent sum (in fact,  even
$\|\cdot
  \|_{\operatorname{loc}}$-convergent sum) of Dirac combs. (For lattices $\Lambda$ with
determinant $\neq 1$ the above formulas must be divided by the
square of\/ $\det(\Lambda)$.)
\end{theorem}
\begin{proof}
Let $\gamma$ be the autocorrelation of $V$. As shown in the proof of
Theorem~\ref{auto}, one has
\[
w(\bs a)=\xi(nk)\sum_{\substack{d=1\\[.5mm]\text{\scriptsize $d$
      squarefree}\\[.5mm]\bs a \in d^k\Lambda}}^{\infty} \frac{1}{d^{nk}}\prod_{p\mid d}\left
  (1-\frac{2}{p^{nk}}\right )^{-1}\,.
\]
So by Theorem~\ref{auto} and~\cite[Lemma 1]{BMP} one obtains
\[
\gamma=\xi(nk)\sum_{\substack{d=1\\[.5mm]\text{\scriptsize $d$
      squarefree}}}^{\infty} \frac{1}{d^{nk}}\prod_{p\mid d}\left
  (1-\frac{2}{p^{nk}}\right )^{-1}\omega_{d^k\Lambda}\,.
\]
Since $\|\omega_{d^k\Lambda}\|_{\operatorname{loc}}=O(1)$ and the
coefficient of $\omega_{d^k\Lambda}$ is $O(1/d^{nk})$, this sum of
tempered distributions is convergent in the weak*-topology
by~\cite[Lemma 2]{BMP}. By the Poisson summation formula for lattice
Dirac combs~\cite[Eq.~(31)]{BMP}, its term-by-term Fourier transform is
\begin{eqnarray*}
\widehat{\gamma}&=&
\xi(nk)\sum_{\substack{d=1\\[.5mm]\text{\scriptsize $d$
      squarefree}}}^{\infty} \frac{1}{d^{2nk}}\prod_{p\mid d}\left
  (1-\frac{2}{p^{nk}}\right )^{-1}\omega_{\Lambda^*/d^k}\\
&=& \xi(nk)\sum_{\substack{d=1\\[.5mm]\text{\scriptsize $d$
      squarefree}}}^{\infty} \left(\prod_{p\mid d}\frac{1}{p^{2nk}-2p^{nk}}\right)\omega_{\Lambda^*/d^k}\,,
\end{eqnarray*}
which weak*-converges to the diffraction measure of $V$, since the
Fourier transform operator is weak*-continuous. Since $\|\omega_{\Lambda^*/d^k}\|_{\operatorname{loc}}=O(d^{nk})$ and the
coefficient of $\omega_{\Lambda^*/d^k}$ is $O(1/d^{2nk})$, the
weak*-sum is a translation-bounded pure-point measure equal to the
pointwise sum of its terms by~\cite[Lemma 2]{BMP}.\footnote{Note that
  $\widehat{\gamma}$ is even a $\|\cdot
  \|_{\operatorname{loc}}$-convergent sum of Dirac combs. Since
  convergence with respect to the local norm preserves the spectral type, it is thus
  clear that $\widehat{\gamma}$ is a pure-point measure;
  cf.~\cite[Theorem 8.4]{BG}.} This establishes
the series form~\eqref{diff} for the diffraction spectrum.

The explicit values of the intensities can now be calculated as
follows. Let $\bs p$ be a point in $\mathbb Q\Lambda^*$ with
denominator $q$. We can assume that $q$ is $(k+1)$-free, since
otherwise there is no contribution 
to~\eqref{diff} at all. The terms in~\eqref{diff} that contribute to
the intensity at $\bs p$ are those with $d=mq^*$, where $q^*$ is the
squarefree kernel of $q$ and  $m\in\mathbb N$ is squarefree and
coprime to $q$. Thus the intensity at $\bs p$ is
\[
\xi(nk)\prod_{p\mid q}\frac{1}{p^{2nk}-2p^{nk}}\sum_{\substack{m=1\\[.5mm]\text{\scriptsize $m$
      squarefree}\\[.5mm] (m,q)=1}}^{\infty}\prod_{p\mid m}\frac{1}{p^{2nk}-2p^{nk}}\,.
\]
Using the Euler products in~\eqref{zetafn} and~\eqref{xi} this
simplifies to
\begin{eqnarray*}
&&\xi(nk)\prod_{p\mid q}\frac{1}{p^{2nk}-2p^{nk}}\prod_{p\nmid
  q}\left(1+\frac{1}{p^{2nk}-2p^{nk}}\right)\\
&=&\xi(nk)\prod_{p\mid
  q}\frac{1}{p^{2nk}}\left(1-\frac{2}{p^{nk}}\right)^{-1}\prod_{p\nmid
q}\left(1-\frac{1}{p^{nk}}\right)^2\left(1-\frac{2}{p^{nk}}\right)^{-1}\\
&=&\frac{1}{\zeta^2(nk)}\prod_{p\mid q}\frac{1}{p^{2nk}}\left(1-\frac{1}{p^{nk}}\right)^{-2}\,,
\end{eqnarray*}
which agrees with~\eqref{intensity}.
\end{proof}

One explicitly sees that $\widehat{\gamma}$ above is fully translation
invariant, with lattice of periods $\Lambda^*$, in accordance
with Theorem~1 of~\cite{B}. Moreover, since the action of the group of
automorphisms of $\Lambda^*$,
$\operatorname{Aut}(\Lambda^*)\simeq {\rm GL}(n,\mathbb Z)$, on $\mathbb
Q\Lambda^*$ preserves the denominator, $\widehat{\gamma}$ is
$(\Lambda^*\rtimes \operatorname{Aut}(\Lambda^*))$-symmetric. In
particular, both $\widehat{\gamma}$ and the set $V$ itself are
${\rm GL}(n,\mathbb Z)$-symmetric.

\section{Improving the error terms}\label{errorterm}

What has kept the error term large in the argument as we have presented
it so far is the last term of \eqref{error}, with $|\mathcal P|$ in the
exponent of $S$ (in the second component of the minimum).  This arose
from the $O(1)$ error term in \eqref{innersum}, when \eqref{innersum}
was substituted for the inner sum in \eqref{VPcount}.  The $O(1)$
error term was not even a boundary effect: it was caused solely by
lattices whose determinants are much larger than the volume of the
region in which points are being counted.  The result of this was to
put the burden of keeping the last term of \eqref{error} small onto $P$
(which occurs in the first component of the minimum), causing an increase
in the error due to the tail of the \mbox{$\zeta$-function} product.
Mirsky's idea in \cite{Mir2} and \cite{Mir1} was to show that the terms
with some $d_i$ large contribute a negligible amount to \eqref{VPcount}
and can be discarded before the substitution of \eqref{innersum} is made.
The remaining terms have the individual $d_i$'s so well bounded that the
second component of the minimum can take over the role of providing a
respectable error term, freeing $P$ to be assigned a much larger value and
thus reducing the size of the tail of the \mbox{$\zeta$-function}
product. 

Let  $r\in\mathbb
Z^+$,  let $\bs p_1,\dots,\bs p_r\in\Lambda$, and let $m_1,\dots,m_r\in\mathbb N$. Define the symbol $E\left(\substack{m_1,\dots,m_{r}\\[.5mm]
    \bs p_1,\dots,\bs p_{r}  }\right)$ as $1$ or $0$ according to the
system of congruences in $\bs t \in\Lambda$,
\begin{equation}\label{system}
\bs t+\bs p_i\in m_i\Lambda\quad (1\le i\le r)\,,
\end{equation}
being solvable or not. Further, for a positive real number $R$ and a point
$\bs x\in\mathbb R^n$, let
$T\Big(\bs x;R;\substack{m_1,\dots,m_{r}\\[.5mm] \bs p_1,\dots,\bs
  p_{r}}\Big)$ denote the number of points $\bs t\in\Lambda$ such that
\begin{equation*}
\begin{array}[t]{c}
\bs t\in B_R(\bs x)\,,\\
\bs t+\bs p_{i}\in m_{i}\Lambda\quad(1\le i\le r)\,.
\end{array}
\end{equation*}
We denote by $[m_1,\dots,m_r]$ the least common
multiple of $m_1,\dots,m_r$.  Further, $(m_i,m_j)$ denotes the greatest
common divisor of $m_i$ and $m_j$. For brevity, let $c(\bs l)$ denote
the $1$-content of a nonzero point $\bs
l\in\Lambda$. 

\begin{lemma}\label{solubility}
The system~\eqref{system} of congruences is soluble if and only if
$$
(m_i,m_j)\mid c(\bs p_i-\bs p_j)\quad (1\le i<j\le r)\,.
$$
In the case of solubility, the solutions form precisely one residue
class $$\pmod {[m_1,\dots,m_r]\Lambda}\,.$$
\end{lemma}
\begin{proof}
This is an immediate consequence of~\cite[Lemma 1]{Mir1} applied to each coordinate
with respect to a basis of $\Lambda$.
\end{proof}

If $m_1,\dots,m_r$ are pairwise coprime, the last result boils
down to the Chinese Remainder Theorem~\eqref{CRT}. In fact, only this
special case will be needed in Theorem~\ref{VPpatches2}
below. However, since the
subsequent lemmas may be of independent interest, we prefer to stick to the
general case.

\begin{lemma}\label{lemT}
$$
T\Big(\bs x;R;\substack{m_1,\dots,m_{r}\\[.5mm] \bs p_1,\dots,\bs
  p_{r}}\Big)=R^nv_n\,\frac{E\left(\substack{m_1,\dots,m_{r}\\[.5mm]
    \bs p_1,\dots,\bs p_{r}  }\right)}{[m_1,\dots,m_{r}]^n}+O\left(\frac{R^{n-1}}{[m_1,\dots,m_r]^{n-1}}\right)+O(1)\,,
$$
where the implied $O$-constants depend only on $\Lambda$. 
\end{lemma}
\begin{proof}
This is an immediate consequence of Lemma~\ref{solubility} together
with~\eqref{LambdaCount} applied to the lattice $[m_1,\dots,m_r]\Lambda$.
\end{proof}

\begin{lemma}\label{lemE}
$$
\frac{E\left(\substack{m_1,\dots,m_{r}\\[.5mm]
    \bs p_1,\dots,\bs p_{r}  }\right)}{[m_1,\dots,m_r]}\le
\frac{K}{m_1\cdots m_r}\,.
$$
where $K$ depends only on $r,\bs p_1,\dots,\bs p_{r}$.
\end{lemma}
\begin{proof}
This follows along the same lines as Lemma~3 of~\cite{Mir1} by
employing Lemma~\ref{solubility} instead of~\cite[Lemma 1]{Mir1}. 
\end{proof}

For points $\bs p_1,\dots,\bs p_r\in\Lambda$ and positive real numbers
$R$ and $\alpha$, denote by $L(\bs x;R;\bs p_1,\dots,\bs p_r;\alpha)$ the cardinality of
systems $(\bs t,a_1,\dots,a_r)$ of lattice points $\bs t\in\Lambda$ and
numbers $a_1,\dots,a_r\in\mathbb N$ such that
\begin{equation*}
\begin{array}[t]{c}
\bs t\in B_R(\bs x)\,,\\
\bs t+\bs p_{i}\in a_{i}^{k}\Lambda\quad(1\le i\le r)\,,\\
a_1\cdots a_r>R^{\alpha}\,.
\end{array}
\end{equation*}

\begin{lemma}\label{freelemma}
$$L(\bs x;R;\bs p_1,\dots,\bs
p_r;\alpha)=O\left(R^{n-\alpha(nk-1)+\epsilon}\right)+O\left(R^{n-1+\frac{2}{nk+1}+\epsilon}\right)\,,$$
where the implied $O$-constants depend only on $\Lambda,k,r,\bs p_1,\dots,\bs
p_{r}$.
\end{lemma} 
\begin{proof}
The proof is by induction on $r$. For $r=1$, we can apply
Lemma~\ref{lemT} to obtain 
\begin{eqnarray*}
L(\bs x;R;\bs
p_1;\alpha)&=&\sum_{\substack{\bs t\in B_R(\bs x)\\[.5mm]
    \bs t+\bs p_1\in a_1^k
    \Lambda\\[.5mm]a_1>R^{\alpha}}} 1\,\,=\,\,\sum_{\substack{a_1>R^{\alpha}\\[.5mm] a_1<
    ((R+\|\bs p_1\|)/\lambda)^{1/k}}}\hspace{2mm}T\Big(\bs x;R;\substack{a_1^k\\[.5mm]
    \bs p_1}\Big)\\
&=&\sum_{\substack{a_1>R^{\alpha}\\[.5mm] a_1<
    ((R+\|\bs p_1\|)/\lambda)^{1/k}}}\Big(\frac{R^nv_n}{a_1^{nk}}+O\Big(\frac{R^{n-1}}{a_1^{(n-1)k}}\Big)+O(1)\Big)\\&=&O\left(R^{n-\alpha(nk-1)}\right)+O\left(R^{n-1}\log R\right)+O(R^{1/k})\,,
\end{eqnarray*}
where there is no middle term when $n=1$ and the logarithm in the
middle term is only needed in the case $n=2$, $k=1$, when the
corresponding harmonic
series $\sum 1/a_1^{(n-1)k}$ diverges. In all other cases, these
series converge, and the middle
term can be
taken as $O(R^{n-1})$. Thus the lemma holds for $r=1$. Assume now that
the assertion holds for some $r\ge
1$. Let $\beta$ be a
positive real parameter to be fixed later. Writing $a=a_1\cdots
a_{r+1}$, for symmetry reasons one has
$$
L(\bs x;R;\bs p_1,\dots,\bs p_{r+1};\alpha)=O(\sum_{\substack{\bs t\in B_R(\bs
    x)\\[.5mm] \bs t+\bs p_1\in a_1^k\Lambda\\[.5mm]\cdots\\[.5mm]\bs
    t+\bs p_{r+1}\in a_{r+1}^k\Lambda\\[.5mm]a>R^{\alpha}\\[.5mm]\frac{a}{a_1},\dots,\frac{a}{a_{r+1}}\le
  x^{\beta}}}1)+O(\sum_{\substack{\bs t\in B_R(\bs
    x)\\[.5mm] \bs t+\bs p_1\in a_1^k\Lambda\\[.5mm]\cdots\\[.5mm]\bs
    t+\bs p_{r+1}\in
    a_{r+1}^k\Lambda\\[.5mm]a>R^{\alpha}\\[.5mm]a_1\cdots a_r>
  x^{\beta}}}1)=L_1+L_2\,,
$$
say. Employing Lemmas~\ref{lemT} and~\ref{lemE}, one obtains
\begin{eqnarray*}
L_1&=&O(\sum_{\substack{\bs t\in B_R(\bs
    x)\\[.5mm] \bs t+\bs p_1\in a_1^k\Lambda\\[.5mm]\cdots\\[.5mm]\bs
    t+\bs p_{r+1}\in a_{r+1}^k\Lambda\\[.5mm]R^{\alpha}<a\le
    R^{\beta(r+1)/r}}}1)\,\,=\,\,O\Big(\sum_{R^{\alpha}<a\le
    R^{\beta(r+1)/r}}T\Big(\bs x;R;\substack{a_1^k,\dots,a_{r+1}^k\\[.5mm]
    \bs p_1,\dots,\bs p_{r+1}}\Big)\Big)\\&=& O\Big(\sum_{R^{\alpha}<a\le
    R^{\beta(r+1)/r}}\Big(R^nv_n\,\frac{E\left(\substack{a_1^k,\dots,a_{r+1}^k\\[.5mm]
    \bs p_1,\dots,\bs p_{r+1}  }\right)}{[a_1^k,\dots,a_{r+1}^k]^n}+R^{n-1}+1\Big)\Big)\\
&=&O\Big(R^n\sum_{a>R^{\alpha}}\frac{1}{(a_1\cdots a_{r+1})^{nk}}\Big)+O\left(R^{n-1+\beta(r+1)/r+\epsilon}\right)\\
&=&O\left(R^{n-\alpha(nk-1)+\epsilon}\right)+O\left(R^{n-1+\beta(r+1)/r+\epsilon}\right)\,.
\end{eqnarray*}
With $\tau$ denoting the ordinary divisor function, one further obtains
\begin{eqnarray*}
L_2&=&O(\sum_{\substack{\bs t\in B_R(\bs
    x)\\[.5mm] \bs t+\bs p_1\in a_1^k\Lambda\\[.5mm]\cdots\\[.5mm]\bs
    t+\bs p_{r}\in
    a_{r}^k\Lambda\\[.5mm]a_1\cdots a_r>
  x^{\beta}}}\hspace{2mm}\sum_{\bs t+\bs p_{r+1}\in a_{r+1}^k\Lambda} 1)\,\,=\,\,O(\sum_{\substack{\bs t\in B_R(\bs
    x)\\[.5mm] \bs t+\bs p_1\in a_1^k\Lambda\\[.5mm]\cdots\\[.5mm]\bs
    t+\bs p_{r}\in
    a_{r}^k\Lambda\\[.5mm]a_1\cdots a_r>
  x^{\beta}}}\tau\left(\|\bs t+\bs p_{r+1}\|/\lambda\right)\\
&=&O(\sum_{\substack{\bs t\in B_R(\bs
    x)\\[.5mm] \bs t+\bs p_1\in a_1^k\Lambda\\[.5mm]\cdots\\[.5mm]\bs
    t+\bs p_{r}\in
    a_{r}^k\Lambda\\[.5mm]a_1\cdots a_r>
  x^{\beta}}}R^{\epsilon})\,\,=\,\,O(R^{\epsilon}L(\bs x;R;\bs p_1,\dots,\bs p_{r};\beta))\\
&=& O\left(R^{n-\beta(nk-1)+2\epsilon}\right)+O\left(R^{n-1+\frac{2}{nk+1}+2\epsilon}\right)\,,
\end{eqnarray*}
by assumption. Setting $\beta:=\frac{r}{rnk+1}$, we obtain 
$$
L(\bs x;R;\bs
p_1,\dots,\bs p_{r+1};\alpha)=O\left(R^{n-\alpha(nk-1)+\epsilon}\right)+O\left(R^{n-1+\frac{2}{nk+1}+2\epsilon}\right)\,,
$$
which proves the lemma. 
\end{proof}

We are now in a position to improve the error term of 
Lemma~\ref{VPpatches}. 

\begin{theorem}\label{VPpatches2}
Let $\mathcal P$ be a finite subset of $\Lambda$, $m\in\mathbb N$,
$\bs m\in\Lambda$, $P$ be a natural number coprime to $m$ and $\bs
x\in\mathbb R^n$. Then
\[|L(V_P;\mathcal P,\emptyset)\cap(\bs m+m\Lambda)\cap B_R(\bs x)|\]
is 
\begin{equation*}
\frac{R^nv_n}{m^n}\prod_{p\mid P}\biggl(1-\frac{|\mathcal
  P/p^k\Lambda|}{p^{nk}}\biggr)+O\left(R^{n-1+\frac{2}{nk+1}+\epsilon}\right)\,,
\end{equation*}
where  the
$O$-constant depends only on $\Lambda$, $k$ and $\mathcal P$.
\end{theorem}

\begin{proof}
This follows from the following modification of the proof of
Lemma~\ref{VPpatches}. We shall also use the notation from that
proof. It suffices to show
that~\eqref{VPcount} is   
\begin{equation*}
\frac{R^nv_n}{m^nQ^{nk}}\prod_{p\mid P/Q}\biggl(1-\frac{r}{p^{nk}}\biggr)+O\left(R^{n-1+\frac{2}{nk+1}+\epsilon}\right)\,.
\end{equation*}
To this end, divide~\eqref{VPcount} as $C_1+C_2$, where
\begin{equation*}
C_1= 
\begin{array}[t]{c}
\displaystyle\sum_{d_1\mid P/Q}\hspace{1mm}\sum_{d_2\mid P/Q}\cdots\sum_{d_r\mid P/Q}\\
\text{\scriptsize$d_1\cdots d_r\le R^{\frac{1}{nk}}$}
\end{array}
\mu(d_1\cdots d_r)\hspace{-4mm}
\sum_{\substack{\bs t\in\Lambda\cap B_R(\bs x)\\[.5mm]
\bs t\in\bs q+mQ^k\Lambda\\[1mm]
\bs t\in -\bs p_i+d_i^k\Lambda}}
1,
\end{equation*}
and $C_2$ consists of the terms with $d_1\cdots d_r>
R^{\frac{1}{nk}}$. By Lemma~\ref{freelemma}, 
$$
C_2=O\left(R^{n-\frac{nk-1}{nk}+\epsilon}\right)+O\left(R^{n-1+\frac{2}{nk+1}+\epsilon}\right)=O\left(R^{n-1+\frac{2}{nk+1}+\epsilon}\right)\,.
$$
One further obtains
\begin{eqnarray*}
C_1&=&\begin{array}[t]{c}
\displaystyle\sum_{d_1\mid P/Q}\hspace{1mm}\sum_{d_2\mid P/Q}\cdots\sum_{d_r\mid P/Q}\\
\text{\scriptsize$d_1\cdots d_r\le R^{\frac{1}{nk}}$}
\end{array}
\mu(d_1\cdots d_r)
\sum_{\substack{\bs t\in\Lambda\cap B_R(\bs x)\\[.5mm]
\bs t\in\bs q+mQ^k\Lambda\\[1mm]
\bs t\in -\bs p_i+d_i^k\Lambda}}
1\\
&=&\begin{array}[t]{c}
\displaystyle\sum_{d_1\mid P/Q}\hspace{1mm}\sum_{d_2\mid P/Q}\cdots\sum_{d_r\mid P/Q}\\
\text{\scriptsize$d_1\cdots d_r\le R^{\frac{1}{nk}}$}
\end{array}
\mu(d_1\cdots d_r)T\Big(\bs
x;R;\substack{d_1^k,\dots,d_{r}^k,mQ^k\\[.5mm]\bs p_1,\dots,\bs
  p_{r},\bs q\hphantom{\,\,Q^k}}\Big)
\end{eqnarray*}
Since the $d_1,\dots,d_r$ are pairwise coprime and since $(mQ,d)=1$,
Lemma~\ref{lemT} in conjunction with Lemma~\ref{solubility} shows that
\begin{eqnarray*}
T\Big(\bs
x;R;\substack{d_1^k,\dots,d_{r}^k,mQ^k\\[.5mm]\bs p_1,\dots,\bs
  p_{r},\bs q\hphantom{\,\,Q^k}}\Big)&=&\frac{R^nv_n}{m^n(dQ)^{nk}}+O\left(\frac{R^{n-1}}{m^{n-1}(dQ)^{(n-1)k}}\right)+O(1)\\&=&\frac{R^nv_n}{m^n(dQ)^{nk}}+O(R^{n-1})\,.
\end{eqnarray*}
Just as in the proof of Lemma~\ref{VPpatches}, substituting this in
the above expression for $C_1$ and removing the condition $d\le R^{\frac{1}{nk}}$ from the sum over
$R^nv_n/(m^n(dQ)^{nk})$ gives the main term
$$
\frac{R^nv_n}{m^nQ^{nk}}\prod_{p\mid P/Q}\biggl(1-\frac{r}{p^{nk}}\biggr)
$$
The error from the extra terms included in
the extended multiple sum is
$$
O\biggl(R^n\sum_{d>R^{\frac{1}{nk}}}\frac{\tau_r(d)}{d^{nk}}\biggr)=O(R^{n-1+\frac{1}{nk}+\epsilon})=O\left(R^{n-1+\frac{2}{nk+1}+\epsilon}\right)\,.
$$
Similarly, the sum over the error term can be seen to be
$O(R^{n-1+\frac{1}{nk}+\epsilon})$. Altogether, this proves the assertion.
\end{proof}

One can now employ Theorem~\ref{VPpatches2} instead of
Lemma~\ref{VPpatches} to see that the
error terms in Corollary~\ref{cor} and Theorem~\ref{density} of the
form $O(R^n/(\log R)^{nk-1})$ can indeed be improved to 
$$O(R^{n-1+\frac{2}{nk+1}+\epsilon})\,.$$ More precisely,
Theorem~\ref{VPpatches2} allows one to choose $P$ as large as the
product of primes less than $R^{1/(nk+1)}$ (instead of $\log R$) in the modified proofs.

\appendix
\renewcommand{\theequation}{A\arabic{equation}}
\setcounter{equation}{0}

\section{Facts from number theory}
We have used a number of standard facts from number theory in this paper,
which we collect here with proper references for convenience.

The \emph{inclusion-exclusion principle} says that if we have
a set of $N$ elements and a list of properties $P_1,P_2,\ldots$,
with $N_i$ elements having property $P_i$, $N_{ij}$ having both
properties $P_i$ and $P_j$, and so on, then the number of elements
having none of the properties is
\begin{equation}\label{in-ex}
N-N_1-N_2-\cdots+N_{12}+\cdots-N_{123}-\cdots\quad\mbox{\cite[Thm.~260]{HW}.}
\end{equation}
The \emph{M\"obius function} is defined for $m\in\mathbb N$ by
\[\mu(m):=
\begin{cases}
1 &\text{when $m=1$},\\
(-1)^r&\text{when $m$ is the product of $r$ distinct primes},\\
0 &\text{otherwise},
\end{cases}\]
and has the property that, for any $m\in\mathbb N$,
\begin{equation}\label{Mobius}
\sum_{d\mid m}\mu(d)=
\begin{cases}
1&\text{if $m=1$},\\0&\text{otherwise},
\end{cases}
\quad\mbox{\cite[Thm.~263]{HW}, \cite[Thm.~6.3.1]{H}}
\end{equation}
(the basis of the M\"obius inversion formula), derived by applying
the inclusion-exclusion principle to the singleton set $\{m\}$ with
the property $P_i$ being divisibility by the $i$th prime. 

The \emph{Riemann $\zeta$-function} is defined for $\Re(s)>1$ by
\begin{equation}\label{zetafn}
\zeta(s):=\sum_{m=1}^\infty\frac{1}{m^s}=
\prod_{\mbox{\scriptsize$p$ prime}}\left(1-\frac{1}{p^s}\right)^{-1}
\end{equation}
and, as a result of M\"obius inversion,
\[\sum_{m=1}^\infty\frac{\mu(m)}{m^s}=\frac{1}{\zeta(s)}
\quad\mbox{\cite[Thm.~287]{HW}, \cite[\S6.14]{H}.}\]
We needed to approximate partial Euler products, slightly more general than
that on the right of \eqref{zetafn}. Let $r\ge0$ and $s>1$ be fixed. Then
\begin{eqnarray*}
0\ge\log\prod_{p\ge N}\left(1-\frac{r}{p^s}\right)
&=&\sum_{p\ge N}\log\left(1-\frac{r}{p^s}\right)\\
&\ge&\sum_{p\ge N}\frac{\frac{r}{p^s}}{\frac{r}{p^s}-1}\quad\mbox{if $N^s>r$}\\
&\ge&\frac{r}{\frac{r}{N^s}-1}\sum_{p\ge N}\frac{1}{p^s}\\
&\ge&\frac{r}{\frac{r}{N^s}-1}\frac{1}{(N-1)^{s-1}}.
\end{eqnarray*}
Hence, on exponentiating,
\begin{equation}\label{Eulertail}
\prod_{p\ge N}\left(1-\frac{r}{p^s}\right)=1-O(N^{1-s}).
\end{equation}
We also had to estimate the product of primes up to a given bound:
\begin{equation}\label{primeprod}
\prod_{p\le N}p<4^N\quad\mbox{\cite[Thm.~415]{HW}.}
\end{equation}
The \emph{$r$-divisor function} $\tau_r(m)$, for $r\ge2$, is the number
of ways of expressing $m$ as an ordered product of $r$ natural numbers.
The special case $r=2$ is the ordinary divisor function
\[\tau(m):=\tau_2(m)=\sum_{d\mid m}1.\]
which satisfies $\tau(m)=O(m^\epsilon)$ for every $\epsilon>0$
\cite[Thm.~315]{HW}, \cite[Thm.~6.5.2]{H}, from which we deduce that also
\begin{equation}\label{taubound}
\tau_r(m)\le\tau(m)^r=O(m^\epsilon)
\end{equation}
Another divisor sum estimate we have used is
\begin{equation}\label{sum1/d}
\sum_{d\mid m}\frac{1}{d}=O(\log\log m)\quad\mbox{\cite[Thm.~323]{HW}.}
\end{equation}

An arithmetic function $f(m)$ (defined on natural numbers $m$) is called
\emph{multiplicative} if $f(m_1m_2)=f(m_1)f(m_2)$ whenever
$(m_1,m_2)=1$. For example, the functions $\mu(m)$ and
$\tau_r(m)$ are clearly multiplicative. A \emph{Dirichlet series},
$\sum_{m=1}^\infty f(m)/m^s$, whose coefficients $f(m)$ are multiplicative
can be expressed as an \emph{Euler product} over the primes $p$,
\begin{equation}\label{Eulerprod}
\sum_{m=1}^\infty\frac{f(m)}{m^s}=\prod_p\left(1+\frac{f(p)}{p^s}+\frac{f(p^2)}{p^{2s}}+\cdots\right),
\end{equation}
for all values of $s$ for which the sum is absolutely convergent.

An estimate we have used frequently is
\begin{equation}\label{LambdaCount}
|B_\rho({\boldsymbol x})\cap\Lambda|=\frac{\rho^nv_n}{\det\Lambda}+
O\left(\left(\frac{\rho^n}{\det\Lambda}\right)^{1-1/n}\right)+O(1),
\end{equation}
approximating the number of points of an $n$-dimensional lattice
$\Lambda$ in a large ball $B_\rho({\bs x})$ (the last error term being
required only when $\det\Lambda$ is bigger than $\rho^n$). This is
obtained by dividing $B_\rho({\bs x})$ into fundamental regions for
$\Lambda$, each of volume $\det\Lambda$ and containing one point of
$\Lambda$, with the error terms arising from fundamental regions that
overlap the boundary of $B_\rho({\bs x})$. The \mbox{$O$-constants}
depend on the shape of $\Lambda$, but not on its size, and are independent
of $\bs x$. A more precise version is given as Proposition~1 of \cite{BMP}.

We have also made much use of the Chinese Remainder Theorem in the form that
if ${\bs m}_1,{\bs m}_2,\ldots,{\bs m}_r$ are points of a lattice $\Lambda$
and $m_1,m_2,\ldots,m_r$ is a set of natural numbers that are
pairwise coprime, then there is a point $\bs t\in\Lambda$ such that for all
points $\bs x\in\Lambda$
\begin{equation}\label{CRT}
{\bs x}\equiv{\bs m}_i\mbox{ (mod~$m_i\Lambda$)  for $i=1,\ldots, r$}
  \Longleftrightarrow {\bs x}\equiv{\bs t}  \mbox{ (mod~$m_1\cdots m_r\Lambda$).}
\end{equation}
This is given as Proposition~2 of \cite{BMP} and is proved by applying
Theorem~2.7.2 of \cite{H} (or Theorem~121 of \cite{HW}) to each coordinate
relative to a basis of $\Lambda$.

We needed a simple upper bound for binomial coefficients.
By comparison with $\int_1^s\log x\,dx$, we have, for $s\in\mathbb N$,
\[\log s!\ge s\log s-s+1\]
(a weak, one-sided version of Stirling's formula), so $s!>(s/e)^s$ and hence
\begin{equation}\label{binomial}
\binom{m}{s}=\frac{m(m-1)\cdots(m-s+1)}{s!}\le\frac{m^s}{s!}<\left(\frac{em}{s}\right)^s.
\end{equation}

\section*{Acknowledgements}
The first author gives special thanks to Daniel Lenz, who brought the question of the entropy
of the visible points to his attention and with whom he had helpful
discussions since, and to Igor Shparlinski for some useful hints on
technique.

The second author is grateful to Michael Baake for the opportunity to finish the
manuscript and for a number of helpful discussions, and to J\"org
Br\"udern for clarifying discussions and for pointing out that the
work of Tsang on squarefree numbers extends to the case of $k$-free
numbers. It is his pleasure to thank Christoph
Richard for several useful hints and for a guide to the literature
of the variational principle in particular. The hospitality of
the Erwin Schr\"odinger Institute in Vienna is gratefully
acknowledged. This work was
supported by the German Research Council (DFG), within the CRC 701.

\end{document}